\documentclass[12pt,reqno,twoside]{amsart}%

\usepackage{amsmath,amsthm,amscd,amsthm,upref,indentfirst}
\usepackage{amsfonts,mathrsfs}

\usepackage{amssymb,amsbsy,bm}
\usepackage{graphicx}
\usepackage{times}
\usepackage{amsaddr}
\usepackage{soul}

\usepackage[us,long,24hr]{datetime}

\usepackage{mathabx}

\usepackage[usenames,dvipsnames]{color}

\theoremstyle{plain}
\newtheorem{theorem}{Theorem}
\newtheorem{assertion}[theorem]{Assertion}

\newtheorem{proposition}[theorem]{Proposition}

\theoremstyle{definition}
\newtheorem{definition}[theorem]{Definition}

\newtheorem{corollary}[theorem]{Corollary}

\theoremstyle{remark}
\newtheorem{remark}[theorem]{Remark}

\newtheorem{example}[theorem]{Example}
\numberwithin{equation}{section}
\numberwithin{theorem}{section}

\usepackage{exscale}
\usepackage[scaled=0.86]{helvet}
\renewcommand{\mathfrak}[1]{{\textbf{\upshape #1}}}
\renewcommand{\mathbf}{\bm}
\renewcommand{\emph}[1]{\textrm{{\upshape #1}}}
\usepackage[scr=euler,scrscaled=1.15,bb=txof,bbscaled=1.16,cal=boondox,calscaled=1.15]{mathalfa}
\renewcommand{\mathit}[1]{\mathscr #1}
\renewcommand{\mathtt}[1]{\scalebox{1}{\bfseries \texttt{\upshape #1}}}

\usepackage{float}

\usepackage[justification=centering]{caption}

\numberwithin{equation}{section}
\numberwithin{theorem}{section}

\usepackage[normalem]{ulem}
\usepackage[square,comma]{natbib}

\usepackage{mparhack}

\usepackage[verbose]{backref}
\backrefsetup{verbose=false}
\renewcommand*{\backref}[1]{}
\renewcommand*{\backrefalt}[4]{[{\tiny%
    \ifcase #1 \textsl{Not cited}%
          \or \textsl{Cited on page}~\textcolor{BrickRed}{#2}%
          \else \textsl{Cited on pages}~\textcolor{BrickRed}{#2}%
    \fi%
    }]}

\usepackage{epigraph}

\usepackage{fancyhdr}
\author{\small\scshape S\lowercase{teven} D\lowercase{uplij}}

\address{
Center for Information Technology (WWU IT),
Universit\"at M\"unster,
R\"ontgenstrasse 7-13\\
D-48149 M\"unster,
Deutschland}
\email{\small \sf douplii@uni-muenster.de;
sduplij@gmail.com;
https://ivv5hpp.uni-muenster.de/u/douplii}
\title{\large\bfseries\scshape
P\lowercase{olyadization of algebraic structures}}

\date{\textit{of start} May 14, 2022. \textit{Date}:
\textit{of completion} July 29, 2022.
\newline
\mbox{}\hskip 1.16em
\textit{Total}:
40
references
}

\renewcommand{\refname}{\textsc{References}}

\let\origsection\section
\renewcommand{\section}[1]{\sectionmark{#1}\origsection{#1}}
\let\origsubsection\subsection
\renewcommand{\subsection}[1]{\subsectionmark{#1}\origsubsection{#1}}

\makeatletter
\renewenvironment{thebibliography}[1]{%
  \@xp\origsection\@xp*\@xp{\refname}%
  \normalfont\footnotesize\labelsep .9em\relax
  \renewcommand\theenumiv{\arabic{enumiv}}\let\p@enumiv\@empty
  \vspace*{-5pt}
  \list{\@biblabel{\theenumiv}}{\settowidth\labelwidth{\@biblabel{#1}}%
    \leftmargin\labelwidth \advance\leftmargin\labelsep
    \usecounter{enumiv}}%
  \sloppy \clubpenalty\@M \widowpenalty\clubpenalty
  \sfcode`\.=\@m
}{%
  \def\@noitemerr{\@latex@warning{Empty `thebibliography' environment}}%
  \endlist
}
\makeatother


\begin{document}
\mbox{}

\mbox{}

\begin{abstract}

\noindent A generalization of the semisimplicity concept for polyadic
algebraic structures is proposed. If semisimple structures can be presented in
block diagonal matrix form (resulting in the Wedderburn decomposition), a
general form of polyadic structures is given by block-shift matrices. We
combine these forms to get a general shape of semisimple nonderived polyadic
structures (\textquotedblleft double\textquotedblright\ decomposition of two kinds).

We then introduce the polyadization concept (a \textquotedblleft polyadic
constructor\textquotedblright) according to which one can construct a
nonderived polyadic algebraic structure of any arity from a given binary
structure. The polyadization of supersymmetric structures is also discussed.
The \textquotedblleft deformation\textquotedblright\ by shifts of operations
on the direct power of binary structures is defined and used to obtain a
nonderived polyadic multiplication. Illustrative concrete examples for the new
constructions are given.

\end{abstract}
\maketitle

\thispagestyle{empty}


\mbox{}
\tableofcontents
\newpage

\pagestyle{fancy}

\addtolength{\footskip}{15pt}

\renewcommand{\sectionmark}[1]{%
\markboth{
{ \scshape #1}}{}}

\renewcommand{\subsectionmark}[1]{%
\markright{
\mbox{\;}\\[5pt]
\textmd{#1}}{}}

\fancyhead{}
\fancyhead[EL,OR]{\leftmark}
\fancyhead[ER,OL]{\rightmark}
\fancyfoot[C]{\scshape -- \textcolor{BrickRed}{\thepage} --}

\renewcommand\headrulewidth{0.5pt}
\fancypagestyle {plain1}{ %
\fancyhf{}
\renewcommand {\headrulewidth }{0pt}
\renewcommand {\footrulewidth }{0pt}
}

\fancypagestyle{plain}{ %
\fancyhf{}
\fancyhead[C]{\scshape S\lowercase{teven} D\lowercase{uplij} \hskip 0.7cm \MakeUppercase{Polyadic Hopf algebras and quantum groups}}
\fancyfoot[C]{\scshape - \thepage  -}
\renewcommand {\headrulewidth }{0pt}
\renewcommand {\footrulewidth }{0pt}
}

\fancypagestyle{fancyref}{ %
\fancyhf{} 
\fancyhead[C]{\scshape R\lowercase{eferences} }
\fancyfoot[C]{\scshape -- \textcolor{BrickRed}{\thepage} --}
\renewcommand {\headrulewidth }{0.5pt}
\renewcommand {\footrulewidth }{0pt}
}

\fancypagestyle{emptyf}{
\fancyhead{}
\fancyfoot[C]{\scshape -- \textcolor{BrickRed}{\thepage} --}
\renewcommand{\headrulewidth}{0pt}
}
\mbox{}
\vskip 5cm
\thispagestyle{emptyf}

\section{\textsc{Introduction}}

\epigraph{I am no poet, but if you think for yourselves, as I proceed,\\ the facts will form a poem in your minds.
}{\textit{``The Life and Letters of Faraday'' (1870) by Bence Jones}\newline\textsc{Michael Faraday}}

The concept of simple and semisimple rings, modules, and algebras (see, e.g.,
\cite{erd/hol,hungerfold,lambek,rotman}) plays a crucial role in the
investigation of Lie algebras and representation theory
\cite{cur/rei,ful/har,knapp}, as well as in category theory
\cite{har70,kno2006,sim77}.

Here we first propose a generalization of this concept for polyadic algebraic
structures \cite{duplij2022}, which can also be important, e.g. in the operad
theory \cite{mar/shn/sta,lod/val} and nonassociative structures
\cite{zhe/sli/she,sab/sbi/she}. If semisimple structures can be presented in
the block-diagonal matrix form (resulting to the Wedderburn decomposition
\cite{wed1908,herstein,lam1991}), a corresponding general form for polyadic
rings can be decomposed to a kind of block-shift matrices \cite{nik84a}. We
combine these forms and introduce a general shape of semisimple polyadic
structures, which are nonderived in the sense that they cannot be obtained as
a successive composition of binary operations, which can be treated as a
polyadic (\textquotedblleft double\textquotedblright) decomposition.

Second, going in the opposite direction, we define the polyadization concept
(\textquotedblleft polyadic constructor\textquotedblright) according to which
one can construct a nonderived polyadic algebraic structure of any arity from
a given binary structure. Then we briefly describe supersymmetric structure polyadization.

Third, we propose operations \textquotedblleft deformed\textquotedblright\ by
shifts to obtain a nonderived $n$-ary multiplication on the direct power of
binary algebraic structures.

For these new constructions some illustrative concrete examples are given.

\section{\textsc{Preliminaries}}

We use notation from \cite{duplij2022,dup2022}. In brief, a (one-set)
\textsl{polyadic algebraic structure} $\mathcal{A}$ is a set $A$ closed with
respect to polyadic operations (or $n$\textsl{-ary multiplication})
$\mu^{\left[  n\right]  }:A^{n}\rightarrow A$ ($n$\textsl{-ary magma}). We
denote \textsl{polyads} \cite{pos} by bold letters $\mathbf{a}=\mathbf{a}%
^{\left(  n\right)  }=\left(  a_{1},\ldots,a_{n}\right)  $, $a_{i}\in A$. A
\textsl{polyadic zero} is defined by $\mu^{\left[  n\right]  }\left[
\mathbf{a}^{\left(  n-1\right)  },z\right]  =z$, $z\in A$, $\mathbf{a}%
^{\left(  n-1\right)  }\in A^{n-1}$, where $z$ can be on any place. A
(positive) \textsl{polyadic power} $\ell_{\mu}\in\mathbb{N}$ is
$a^{\left\langle \ell_{\mu}\right\rangle }=\left(  \mu^{\left[  n\right]
}\right)  ^{\circ\ell_{\mu}}\left[  a^{\ell_{\mu}\left(  n-1\right)
+1}\right]  ,\ \ \ \ a\in A$. An element of a polyadic algebraic structure $a$
is called $\ell_{\mu}$-\textsl{nilpotent} (or simply \textsl{nilpotent} for
$\ell_{\mu}=1$), if there exist $\ell_{\mu}$ such that $a^{\left\langle
\ell_{\mu}\right\rangle }=z$. A \textsl{polyadic (or }$n$\textsl{-ary)
identity} (or neutral element) is defined by $\mu^{\left[  n\right]  }\left[
a,e^{n-1}\right]  =a,\ \ \ \ \forall a\in A$, where $a$ can be on any place in
the l.h.s. A one-set polyadic algebraic structure $\left\langle A\mid
\mu^{\left[  n\right]  }\right\rangle $ is \textsl{totally associative}, if
$\left(  \mu^{\left[  n\right]  }\right)  ^{\circ2}\left[  \mathbf{a}%
,\mathbf{b},\mathbf{c}\right]  =\mu^{\left[  n\right]  }\left[  \mathbf{a}%
,\mu^{\left[  n\right]  }\left[  \mathbf{b}\right]  ,\mathbf{c}\right]
=invariant$, with respect to placement of the internal multiplication on any
of the $n$ places, and $\mathbf{a},\mathbf{b},\mathbf{c}$ are polyads of the
necessary sizes \cite{dup2018a,dup2019}. A \textsl{polyadic semigroup}
$\mathcal{S}^{\left(  n\right)  }$ is a one-set and one-operation structure in
which $\mu^{\left[  n\right]  }$ is totally associative. A polyadic structure
is \textsl{commutative}, if $\mu^{\left[  n\right]  }=\mu^{\left[  n\right]
}\circ\sigma$, or $\mu^{\left[  n\right]  }\left[  \mathbf{a}\right]
=\mu^{\left[  n\right]  }\left[  \sigma\circ\mathbf{a}\right]  $,
$\mathbf{a}\in A^{n}$, for all $\sigma\in S_{n}$.

A polyadic structure is \textsl{solvable}, if for all polyads $\mathbf{b}$,
$\mathbf{c}$ and an element $x$, one can (uniquely) resolve the equation (with
respect to $h$) for $\mu^{\left[  n\right]  }\left[  \mathbf{b},x,\mathbf{c}%
\right]  =a$, where $x$ can be on any place, and $\mathbf{b},\mathbf{c}$ are
polyads of the needed lengths. A solvable polyadic structure is called a
\textsl{polyadic quasigroup} \cite{belousov}. An associative polyadic
quasigroup is called a $n$-\textsl{ary} (or \textsl{polyadic})\textit{
}\textsl{group} \cite{galmak1}. In an $n$-ary group the only solution of
\begin{equation}
\mu^{\left[  n\right]  }\left[  \mathbf{b},\bar{a}\right]  =a,\ \ \ a,\bar
{a}\in A,\ \ \mathbf{b}\in A^{n-1} \label{ck-mgg}%
\end{equation}
is called a \textsl{querelement} of $a$ and denoted by $\bar{a}$ \cite{dor3},
where $\bar{a}$ can be on any place. Any idempotent $a$ coincides with its
querelement $\bar{a}=a$. The relation (\ref{ck-mgg}) can be considered as a
definition of the unary \textsl{queroperation} $\bar{\mu}^{\left(  1\right)
}\left[  a\right]  =\bar{a}$ \cite{gle/gla}. For further details and
definitions, see \cite{duplij2022}.

\bigskip

\section{\textsc{Polyadic semisimplicity}}

In general, simple algebraic structures are building blocks (direct summands)
for the semisimple ones satisfying special conditions (see, e.g.,
\cite{erd/hol,lambek}).

\subsection{Simple polyadic structures}

According to the Wedderburn-Artin theorem (see, e.g.,
\cite{herstein,lam1991,haz/gub}), a ring which is simple (having no two-sided
ideals, except zero and the ring itself) and Artinian (having minimal right
ideals) $\mathcal{R}_{simple}$ is isomorphic to a full $d\times d$ matrix ring%
\begin{equation}
\mathcal{R}_{simple}\cong Mat_{d\times d}^{full}\left(  \mathcal{D}\right)
\label{rm}%
\end{equation}
over a division ring $\mathcal{D}$. As a corollary,%
\begin{equation}
\mathcal{R}_{simple}\cong\operatorname*{Hom}\nolimits_{\mathcal{D}}\left(
V\left(  d\mid\mathcal{D}\right)  ,V\left(  d\mid\mathcal{D}\right)  \right)
\equiv\operatorname*{End}\nolimits_{\mathcal{D}}\left(  V\left(
d\mid\mathcal{D}\right)  \right)  , \label{r}%
\end{equation}
where $V\left(  d\mid\mathcal{D}\right)  $ is a $d$-finite-dimensional vector
space (left module) over $\mathcal{D}$. In the same way, a finite-dimensional
simple associative algebra $\mathcal{A}$ over an algebraically closed field
$\mathcal{F}$ is%
\begin{equation}
\mathcal{A}\cong Mat_{d\times d}^{full}\left(  \mathcal{F}\right)  .
\label{af}%
\end{equation}

In the polyadic case, the structure of a simple Artinian $\left[  2,n\right]
$-ring $\mathcal{R}_{simple}^{\left[  2,n\right]  }$ (with binary addition and
$n$-ary multiplication $\mu^{\left[  n\right]  }$) was obtained in
\cite{nik84a}, where the Wedderburn-Artin theorem for $\left[  2,n\right]
$-rings was proved. So instead of one vector space $V\left(  d\mid
\mathcal{D}\right)  $, one should consider a direct sum of $\left(
n-1\right)  $ vector spaces (over the same division ring $\mathcal{D}$), that
is
\begin{equation}
V_{1}\left(  d_{1}\mid\mathcal{D}\right)  \oplus V_{2}\left(  d_{2}%
\mid\mathcal{D}\right)  \oplus\ldots\ldots\oplus V_{n-1}\left(  d_{n-1}%
\mid\mathcal{D}\right)  , \label{vd}%
\end{equation}
where $V_{i}\left(  d_{i}\mid\mathcal{D}\right)  $ is a $d_{i}$-dimensional
polyadic vector space \cite{dup2019}, $i=1,\ldots,n-1$. Then, instead of
(\ref{r}) we have the cyclic direct sum of homomorphisms%
\begin{align}
\mathcal{R}_{simple}^{\left[  2,n\right]  }  &  \cong\operatorname*{Hom}%
\nolimits_{\mathcal{D}}\left(  V_{1}\left(  d_{1}\mid\mathcal{D}\right)
,V_{2}\left(  d_{2}\mid\mathcal{D}\right)  \right)  \oplus\operatorname*{Hom}%
\nolimits_{\mathcal{D}}\left(  V_{2}\left(  d_{2}\mid\mathcal{D}\right)
,V_{3}\left(  d_{3}\mid\mathcal{D}\right)  \right)  \oplus\ldots\nonumber\\
&  \ldots\oplus\operatorname*{Hom}\nolimits_{\mathcal{D}}\left(
V_{n-1}\left(  d_{n-1}\mid\mathcal{D}\right)  ,V_{1}\left(  d_{1}%
\mid\mathcal{D}\right)  \right)  . \label{rh}%
\end{align}
This means that after choosing a suitable basis in terms of matrices (when the
ring multiplication $\mu^{\left[  n\right]  }$ coincides with the product of
$n$ matrices) we have

\begin{theorem}
The simple polyadic ring $\mathcal{R}_{simple}^{\left[  2,n\right]  }$ is
isomorphic to the $d\times d$ matrix ring (cf. (\ref{rm}))%
\begin{equation}
\mathcal{R}_{simple}^{\left[  2,n\right]  }\cong Mat_{d\times d}^{shift\left(
n\right)  }\left(  \mathcal{D}\right)  =\left\{  \mathrm{M}^{shift\left(
n\right)  }\left(  d\times d\right)  \mid\nu^{\left[  2\right]  },\mu^{\left[
n\right]  }\right\}  , \label{rmm}%
\end{equation}
where $\nu^{\left[  2\right]  }$ and $\mu^{\left[  n\right]  }$ are binary
addition and ordinary product of $n$ matrices, $\mathrm{M}_{d\times d}%
^{shift}$ is the block-shift (traceless) matrix over $\mathcal{D}$ of the form
(which follows from (\ref{rh}))%
\begin{equation}
\mathrm{M}^{shift\left(  n\right)  }\left(  d\times d\right)  =\left(
\begin{array}
[c]{ccccc}%
0 & \mathrm{B}_{1}\left(  d_{1}\times d_{2}\right)  & \ldots & 0 & 0\\
0 & 0 & \mathrm{B}_{2}\left(  d_{2}\times d_{3}\right)  & \ldots & 0\\
0 & 0 & \ddots & \ddots & \vdots\\
\vdots & \vdots & \ddots & 0 & \mathrm{B}_{n-2}\left(  d_{n-2}\times
d_{n-1}\right) \\
\mathrm{B}_{n-1}\left(  d_{n-1}\times d_{1}\right)  & 0 & \ldots & 0 & 0
\end{array}
\right)  , \label{mb}%
\end{equation}
where $\left(  n-1\right)  $ blocks are nonsquare matrices $\mathrm{B}%
_{i}\left(  d^{\prime}\times d^{\prime\prime}\right)  \in Mat_{d^{\prime
}\times d^{\prime\prime}}^{full}\left(  \mathcal{D}\right)  $ over the
division ring $\mathcal{D}$, and $d=d_{1}+d_{2}+\ldots+d_{n-1}$.{}
\end{theorem}

\begin{remark}
The set of the fixed size blocks $\left\{  \mathrm{B}_{i}\left(  d^{\prime
}\times d^{\prime\prime}\right)  \right\}  $ does not form a binary ring,
because $d^{\prime}\neq d^{\prime\prime}$.
\end{remark}

\begin{assertion}
The block-shift matrices of the form (\ref{mb}) are closed with respect to
$n$-ary multiplication and binary addition, and we call them $n$-\textsl{ary
matrices}.
\end{assertion}

Taking distributivity into account we arrive at the polyadic ring structure
(\ref{rmm}).

\begin{corollary}
In the limiting case $n=2$, we have%
\begin{equation}
\mathrm{M}^{shift\left(  n=2\right)  }\left(  d\times d\right)  =\mathrm{B}%
_{1}\left(  d_{1}\times d_{1}\right)  \label{m2}%
\end{equation}
and $d=d_{1}$, giving a binary ring (\ref{rm}).
\end{corollary}

\begin{assertion}
A finite-dimensional simple associative $n$-ary algebra $\mathcal{A}^{\left(
n\right)  }$ over an algebraically closed field $\mathcal{F}$ \cite{car4} is
isomorphic to the block-shift $n$-ary matrix (\ref{mb}) over $\mathcal{F}$%
\begin{equation}
\mathcal{A}^{\left(  n\right)  }\cong Mat_{d\times d}^{shift\left(  n\right)
}\left(  \mathcal{F}\right)  . \label{anf}%
\end{equation}

\end{assertion}

\subsection{Semisimple polyadic structures}

The Wedderburn-Artin theorem for semisimple Artinian rings $\mathcal{R}%
_{semispl}$ states that $\mathcal{R}_{semispl}$ is a finite direct product of
$k$ simple rings, each of which has the form (\ref{rm}). Using (\ref{r}) for
each component, we decompose the $d$-finite-dimensional vector space (left
module) into a direct sum of length $k$%
\begin{equation}
V\left(  d\right)  =W^{\left(  1\right)  }\left(  q^{\left(  1\right)  }%
\mid\mathcal{D}^{\left(  1\right)  }\right)  \oplus W^{\left(  2\right)
}\left(  q^{\left(  2\right)  }\mid\mathcal{D}^{\left(  2\right)  }\right)
\oplus\ldots\oplus W^{\left(  k\right)  }\left(  q^{\left(  k\right)  }%
\mid\mathcal{D}^{\left(  k\right)  }\right)  , \label{vw}%
\end{equation}
where $d=q^{\left(  1\right)  }+q^{\left(  2\right)  }+\ldots+q^{\left(
k\right)  }$. Then, instead of (\ref{r}) we have the isomorphism\footnote{We
enumerate simple components by an upper index in round brackets $\left(
k\right)  $, block-shift components by lower index without brackets, and the
arity is an upper index in square brackets $\left[  n\right]  $.}%
\begin{equation}
\mathcal{R}_{semispl}\cong\operatorname*{End}\nolimits_{\mathcal{D}^{\left(
1\right)  }}W^{\left(  1\right)  }\left(  q^{\left(  1\right)  }%
\mid\mathcal{D}^{\left(  1\right)  }\right)  \oplus\operatorname*{End}%
\nolimits_{\mathcal{D}^{\left(  2\right)  }}W^{\left(  2\right)  }\left(
q^{\left(  2\right)  }\mid\mathcal{D}^{\left(  2\right)  }\right)
\oplus\ldots\oplus\operatorname*{End}\nolimits_{\mathcal{D}^{\left(  k\right)
}}W^{\left(  k\right)  }\left(  q^{\left(  k\right)  }\mid\mathcal{D}^{\left(
k\right)  }\right)  . \label{rw}%
\end{equation}

In a suitable basis the Wedderburn-Artin theorem follows

\begin{theorem}
A semisimple Artinian (binary) ring $\mathcal{R}_{semispl}$ is isomorphic to
the $d\times d$ matrix ring%
\begin{equation}
\mathcal{R}_{semispl}\cong Mat_{q^{\left(  j\right)  }\times q^{\left(
j\right)  }}^{diag\left(  k\right)  }\left(  \mathcal{D}\right)  =\left\{
\mathrm{M}^{diag\left(  k\right)  }\left(  d\times d\right)  \mid\nu^{\left[
2\right]  },\mu^{\left[  2\right]  }\right\}  ,
\end{equation}
where $\nu^{\left[  2\right]  }$ and $\mu^{\left[  2\right]  }$ are binary
addition and binary product of matrices, $\mathrm{M}^{diag\left(  k\right)
}\left(  d\times d\right)  $ are block-diagonal matrices of the form (which
follows from (\ref{rw}))%
\begin{equation}
\mathrm{M}^{diag\left(  k\right)  }\left(  d\times d\right)  =\left(
\begin{array}
[c]{cccc}%
\mathrm{A}^{\left(  1\right)  }\left(  q^{\left(  1\right)  }\times q^{\left(
1\right)  }\right)  & 0 & \ldots & 0\\
0 & \mathrm{A}^{\left(  2\right)  }\left(  q^{\left(  2\right)  }\times
q^{\left(  2\right)  }\right)  & \ddots & \vdots\\
\vdots & \vdots & \ddots & 0\\
0 & 0 & \ldots & \mathrm{A}^{\left(  k\right)  }\left(  q^{\left(  k\right)
}\times q^{\left(  k\right)  }\right)
\end{array}
\right)  , \label{ma}%
\end{equation}
where $k$ square blocks are full matrix rings over division rings
$\mathcal{D}^{\left(  j\right)  }$%
\begin{equation}
\mathrm{A}^{\left(  j\right)  }\left(  q^{\left(  j\right)  }\times q^{\left(
j\right)  }\right)  \in Mat_{q^{\left(  j\right)  }\times q^{\left(  j\right)
}}^{full}\left(  \mathcal{D}^{\left(  j\right)  }\right)  ,\ \ j=1,\ldots
,k,\ \ d=q^{\left(  1\right)  }+q^{\left(  2\right)  }+\ldots+q^{\left(
k\right)  }. \label{aq}%
\end{equation}

\end{theorem}

The same matrix structure has a finite-dimensional semisimple associative
algebra $\mathcal{A}$ over an algebraically closed field $\mathcal{F}$ (see
(\ref{af})). For further details, see, e.g., \cite{herstein,lam1991,haz/gub}.

General properties of semisimple Artinian $\left[  2,n\right]  $-rings were
considered in \cite{nik84a} (for ternary rings, see \cite{lis,pro1982}). Here
we propose a new manifest matrix structure for them.

Thus, our task is to decompose each of the $V_{i}\left(  d_{i}\right)  $, in
(\ref{vd}) into components as in (\ref{vw})%
\begin{equation}
V_{i}\left(  d_{i}\right)  =W_{i}^{\left(  1\right)  }\left(  q_{i}^{\left(
1\right)  }\mid\mathcal{D}^{\left(  1\right)  }\right)  \oplus W_{i}^{\left(
2\right)  }\left(  q_{i}^{\left(  2\right)  }\mid\mathcal{D}^{\left(
2\right)  }\right)  \oplus\ldots\oplus W_{i}^{\left(  k\right)  }\left(
q_{i}^{\left(  k\right)  }\mid\mathcal{D}^{\left(  k\right)  }\right)
,\ \ \ i=1,\ldots,n-1.
\end{equation}

In matrix language this means that each block $\mathrm{B}_{d^{\prime}\times
d^{\prime\prime}}$ from the polyadic ring (\ref{mb}) should have the
semisimple decomposition (\ref{ma}), i.e. be a block-diagonal square matrix of
the same size $p\times p$, where $p=d_{1}=d_{2}=\ldots=d_{n-1}$ and the total
matrix size becomes $d=\left(  n-1\right)  p$. Moreover, all the blocks
$\mathrm{B}$'s should have diagonal blocks $\mathrm{A}$'s of the same size,
and therefore $q^{\left(  j\right)  }\equiv q_{1}^{\left(  j\right)  }%
=q_{2}^{\left(  j\right)  }=\ldots=q_{n-1}^{\left(  j\right)  }$ for all
$j=1,\ldots,k$ and $p=q^{\left(  1\right)  }+q^{\left(  2\right)  }%
+\ldots+q^{\left(  k\right)  }$, where $k$ is the number of semisimple
components. In this way the cyclic direct sum of homomorphisms for the
semisimple polyadic rings becomes (we use different division rings for each
semisimple component as in (\ref{aq}))%
\begin{align}
\mathcal{R}_{semispl}^{\left[  2,n\right]  }  &  \cong\operatorname*{Hom}%
\nolimits_{\mathcal{D}^{\left(  1\right)  }}\left(  W_{1}^{\left(  1\right)
}\left(  q^{\left(  1\right)  }\mid\mathcal{D}^{\left(  1\right)  }\right)
,W_{2}^{\left(  1\right)  }\left(  q^{\left(  1\right)  }\mid\mathcal{D}%
^{\left(  1\right)  }\right)  \right) \nonumber\\
&  \oplus\operatorname*{Hom}\nolimits_{\mathcal{D}^{\left(  2\right)  }%
}\left(  W_{1}^{\left(  2\right)  }\left(  q^{\left(  2\right)  }%
\mid\mathcal{D}^{\left(  2\right)  }\right)  ,W_{2}^{\left(  2\right)
}\left(  q^{\left(  2\right)  }\mid\mathcal{D}^{\left(  2\right)  }\right)
\right)  \oplus\ldots\nonumber\\
&  \ldots\oplus\operatorname*{Hom}\nolimits_{\mathcal{D}^{\left(  k\right)  }%
}\left(  W_{1}^{\left(  k\right)  }\left(  q^{\left(  k\right)  }%
\mid\mathcal{D}^{\left(  k\right)  }\right)  ,W_{2}^{\left(  k\right)
}\left(  q^{\left(  k\right)  }\mid\mathcal{D}^{\left(  k\right)  }\right)
\right) \nonumber\\
&  \oplus\operatorname*{Hom}\nolimits_{\mathcal{D}^{\left(  1\right)  }%
}\left(  W_{2}^{\left(  1\right)  }\left(  q^{\left(  1\right)  }%
\mid\mathcal{D}^{\left(  1\right)  }\right)  ,W_{3}^{\left(  1\right)
}\left(  q^{\left(  1\right)  }\mid\mathcal{D}^{\left(  1\right)  }\right)
\right) \nonumber\\
&  \oplus\operatorname*{Hom}\nolimits_{\mathcal{D}^{\left(  2\right)  }%
}\left(  W_{2}^{\left(  2\right)  }\left(  q^{\left(  2\right)  }%
\mid\mathcal{D}^{\left(  2\right)  }\right)  ,W_{3}^{\left(  2\right)
}\left(  q^{\left(  2\right)  }\mid\mathcal{D}^{\left(  2\right)  }\right)
\right)  \oplus\ldots\nonumber\\
&  \ldots\oplus\operatorname*{Hom}\nolimits_{\mathcal{D}^{\left(  k\right)  }%
}\left(  W_{2}^{\left(  k\right)  }\left(  q^{\left(  k\right)  }%
\mid\mathcal{D}^{\left(  k\right)  }\right)  ,W_{3}^{\left(  k\right)
}\left(  q^{\left(  k\right)  }\mid\mathcal{D}^{\left(  k\right)  }\right)
\right) \nonumber\\
&  \vdots\nonumber\\
&  \oplus\operatorname*{Hom}\nolimits_{\mathcal{D}^{\left(  1\right)  }%
}\left(  W_{n-2}^{\left(  1\right)  }\left(  q^{\left(  1\right)  }%
\mid\mathcal{D}^{\left(  1\right)  }\right)  ,W_{n-1}^{\left(  1\right)
}\left(  q^{\left(  1\right)  }\mid\mathcal{D}^{\left(  1\right)  }\right)
\right) \nonumber\\
&  \oplus\operatorname*{Hom}\nolimits_{\mathcal{D}^{\left(  2\right)  }%
}\left(  W_{n-2}^{\left(  2\right)  }\left(  q^{\left(  2\right)  }%
\mid\mathcal{D}^{\left(  2\right)  }\right)  ,W_{n-1}^{\left(  2\right)
}\left(  q^{\left(  2\right)  }\mid\mathcal{D}^{\left(  2\right)  }\right)
\right)  \oplus\ldots\nonumber\\
&  \ldots\oplus\operatorname*{Hom}\nolimits_{\mathcal{D}^{\left(  k\right)  }%
}\left(  W_{n-2}^{\left(  k\right)  }\left(  q^{\left(  k\right)  }%
\mid\mathcal{D}^{\left(  k\right)  }\right)  ,W_{n-1}^{\left(  k\right)
}\left(  q^{\left(  k\right)  }\mid\mathcal{D}^{\left(  k\right)  }\right)
\right) \nonumber\\
&  \oplus\operatorname*{Hom}\nolimits_{\mathcal{D}^{\left(  1\right)  }%
}\left(  W_{n-1}^{\left(  1\right)  }\left(  q^{\left(  1\right)  }%
\mid\mathcal{D}^{\left(  1\right)  }\right)  ,W_{1}^{\left(  1\right)
}\left(  q^{\left(  1\right)  }\mid\mathcal{D}^{\left(  1\right)  }\right)
\right) \nonumber\\
&  \oplus\operatorname*{Hom}\nolimits_{\mathcal{D}^{\left(  2\right)  }%
}\left(  W_{n-1}^{\left(  2\right)  }\left(  q^{\left(  2\right)  }%
\mid\mathcal{D}^{\left(  2\right)  }\right)  ,W_{1}^{\left(  2\right)
}\left(  q^{\left(  2\right)  }\mid\mathcal{D}^{\left(  2\right)  }\right)
\right)  \oplus\ldots\nonumber\\
&  \ldots\oplus\operatorname*{Hom}\nolimits_{\mathcal{D}^{\left(  k\right)  }%
}\left(  W_{n-1}^{\left(  k\right)  }\left(  q^{\left(  k\right)  }%
\mid\mathcal{D}^{\left(  k\right)  }\right)  ,W_{1}^{\left(  k\right)
}\left(  q^{\left(  k\right)  }\mid\mathcal{D}^{\left(  k\right)  }\right)
\right)  . \label{rh2}%
\end{align}

After choosing a suitable basis we obtain a polyadic analog of the
Wedderburn-Artin theorem for semisimple Artinian $\left[  2,n\right]  $-rings
$\mathcal{R}_{semispl}^{\left[  2,n\right]  }$, which can be called as the
\textsl{double decomposition} (of the \textsl{first kind} or
\textsl{shift-diagonal}).

\begin{theorem}
The semisimple polyadic Artinian ring $\mathcal{R}_{semispl}^{\left[
2,n\right]  }$ (of the first kind) is isomorphic to the $d\times d$ matrix
ring%
\begin{equation}
\mathcal{R}_{semispl}^{\left[  2,n\right]  }\cong Mat_{d\times d}%
^{shift\text{-}diag\left(  n,k\right)  }\left(  \mathcal{D}\right)
=\left\langle \left\{  \mathrm{N}^{shift\text{-}diag\left(  n,k\right)
}\left(  d\times d\right)  \right\}  \mid\nu^{\left[  2\right]  },\mu^{\left[
n\right]  }\right\rangle ,
\end{equation}
where $\nu^{\left[  2\right]  },\mu^{\left[  n\right]  }$ are binary addition
and ordinary product of $n$ matrices, $\mathrm{N}_{d\times d}^{shift\text{-}%
diag\left(  n,k\right)  }$ ($n$ is arity of $\mathrm{N}$'s and $k$ is number
of simple components of $\mathrm{N}$'s) are the block-shift $n$-ary matrices
with block-diagonal square blocks (which follows from (\ref{rh2}))%
\begin{align}
\mathrm{N}^{shift\text{-}diag\left(  n,k\right)  }\left(  d\times d\right)
&  =\left(
\begin{array}
[c]{ccccc}%
0 & \mathrm{B}_{1}^{\left(  k\right)  }\left(  p\times p\right)  & \ldots &
0 & 0\\
0 & 0 & \mathrm{B}_{2}^{\left(  k\right)  }\left(  p\times p\right)  & \ldots
& 0\\
0 & 0 & \ddots & \ddots & \vdots\\
\vdots & \vdots & \ddots & 0 & \mathrm{B}_{n-2}^{\left(  k\right)  }\left(
p\times p\right) \\
\mathrm{B}_{n-1}^{\left(  k\right)  }\left(  p\times p\right)  & 0 & \ldots &
0 & 0
\end{array}
\right)  ,\label{n1}\\
\mathrm{B}_{i}^{\left(  k\right)  }\left(  p\times p\right)   &  =\left(
\begin{array}
[c]{cccc}%
\mathrm{A}_{i}^{\left(  1\right)  }\left(  q^{\left(  1\right)  }\times
q^{\left(  1\right)  }\right)  & 0 & \ldots & 0\\
0 & \mathrm{A}_{i}^{\left(  2\right)  }\left(  q^{\left(  2\right)  }\times
q^{\left(  2\right)  }\right)  & \ddots & \vdots\\
\vdots & \vdots & \ddots & 0\\
0 & 0 & \ldots & \mathrm{A}_{i}^{\left(  k\right)  }\left(  q^{\left(
k\right)  }\times q^{\left(  k\right)  }\right)
\end{array}
\right)  , \label{n2}%
\end{align}
where%
\begin{align}
d  &  =\left(  n-1\right)  p,\\
p  &  =q^{\left(  1\right)  }+q^{\left(  2\right)  }+\ldots+q^{\left(
k\right)  },
\end{align}
and the $k$ square blocks $\mathrm{A}$'s are full matrix rings over the
division rings $\mathcal{D}^{\left(  j\right)  }$%
\begin{equation}
\mathrm{A}_{i}^{\left(  j\right)  }\left(  q^{\left(  j\right)  }\times
q^{\left(  j\right)  }\right)  \in Mat_{q^{\left(  j\right)  }\times
q^{\left(  j\right)  }}^{full}\left(  \mathcal{D}^{\left(  j\right)  }\right)
,\ \ j=1,\ldots,k,\ i=1,\ldots,n-1. \label{aqq}%
\end{equation}

\end{theorem}

\begin{remark}
By analogy with (\ref{m2}), in the limiting case $n=2$, we have in (\ref{n1})
one block $\mathrm{B}_{1}^{\left(  k\right)  }\left(  p\times p\right)  $ only
and (\ref{n2}) gives its standard (binary) semisimple ring decomposition.
\end{remark}

This allows us to introduce another possible double decomposition in the
opposite sequence to (\ref{n1})-(\ref{n2}), we call it \textsl{of second kind}
or \textsl{reverse}, or \textsl{diagonal-shift}. Indeed, in a suitable basis
we first provide the standard block-diagonal decomposition (\ref{ma}), and
then each block obeys the block-shift decomposition (\ref{mb}). Here we do not
write the \textquotedblleft reverse\textquotedblright\ analog of (\ref{rh2})
and arrive directly to

\begin{theorem}
The semisimple polyadic Artinian ring $\mathcal{\hat{R}}_{semispl}^{\left[
2,n\right]  }$ (of the second kind) is isomorphic to the $d\times d$ matrix
ring%
\begin{equation}
\mathcal{\hat{R}}_{semispl}^{\left[  2,n\right]  }\cong Mat_{d\times
d}^{diag\text{-}shift\left(  n,k\right)  }\left(  \mathcal{D}\right)
=\left\langle \left\{  \mathrm{\hat{N}}^{diag\text{-}shift\left(  n,k\right)
}\left(  d\times d\right)  \right\}  \mid\nu^{\left[  2\right]  },\mu^{\left[
n\right]  }\right\rangle ,
\end{equation}
where $\nu^{\left[  2\right]  },\mu^{\left[  n\right]  }$ are binary addition
and ordinary product of $n$ matrices, $\mathrm{\hat{N}}_{d\times
d}^{diag\text{-}shift\left(  n,k\right)  }$ ($n$ is arity of $\mathrm{\hat{N}%
}$'s and $k$ is number of simple components of $\mathrm{\hat{N}}$'s) are the
block-diagonal $n$-ary matrices with block-shift nonsquare blocks%
\begin{align}
&  \mathrm{\hat{N}}^{diag\text{-}shift\left(  n,k\right)  }\left(  d\times
d\right)  =\left(
\begin{array}
[c]{cccc}%
\mathrm{\hat{A}}^{\left(  1\right)  }\left(  q^{\left(  1\right)  }\times
q^{\left(  1\right)  }\right)  & 0 & \ldots & 0\\
0 & \mathrm{\hat{A}}^{\left(  2\right)  }\left(  q^{\left(  2\right)  }\times
q^{\left(  2\right)  }\right)  & \ddots & \vdots\\
\vdots & \vdots & \ddots & 0\\
0 & 0 & \ldots & \mathrm{\hat{A}}^{\left(  k\right)  }\left(  q^{\left(
k\right)  }\times q^{\left(  k\right)  }\right)
\end{array}
\right)  ,\label{nn1}\\
&  \mathrm{\hat{A}}^{\left(  j\right)  }\left(  q^{\left(  j\right)  }\times
q^{\left(  j\right)  }\right) \nonumber\\
&  =\left(
\begin{array}
[c]{ccccc}%
0 & \mathrm{\hat{B}}_{1}^{\left(  j\right)  }\left(  p_{1}^{\left(  j\right)
}\times p_{2}^{\left(  j\right)  }\right)  & \ldots & 0 & 0\\
0 & 0 & \mathrm{\hat{B}}_{2}^{\left(  j\right)  }\left(  p_{2}^{\left(
j\right)  }\times p_{3}^{\left(  j\right)  }\right)  & \ldots & 0\\
0 & 0 & \ddots & \ddots & \vdots\\
\vdots & \vdots & \ddots & 0 & \mathrm{\hat{B}}_{n-2}^{\left(  j\right)
}\left(  p_{n-2}^{\left(  j\right)  }\times p_{n-1}^{\left(  j\right)
}\right) \\
\mathrm{\hat{B}}_{n-1}^{\left(  j\right)  }\left(  p_{n-1}^{\left(  j\right)
}\times p_{1}^{\left(  j\right)  }\right)  & 0 & \ldots & 0 & 0
\end{array}
\right)  , \label{nn2}%
\end{align}
where%
\begin{align}
q^{\left(  j\right)  }  &  =p_{1}^{\left(  j\right)  }+p_{2}^{\left(
j\right)  }+\ldots+p_{n-1}^{\left(  j\right)  },\\
d  &  =q^{\left(  1\right)  }+q^{\left(  2\right)  }+\ldots+q^{\left(
k\right)  },
\end{align}
and the $\left(  n-1\right)  k$ blocks $\mathrm{\hat{B}}$'s are nonsquare
matrices over the division rings $\mathcal{D}^{\left(  j\right)  }$%
\begin{equation}
\mathrm{\hat{B}}_{i}^{\left(  j\right)  }\left(  p_{i}^{\left(  j\right)
}\times p_{i+1}^{\left(  j\right)  }\right)  \in Mat_{p_{i}^{\left(  j\right)
}\times p_{i+1}^{\left(  j\right)  }}^{full}\left(  \mathcal{D}^{\left(
j\right)  }\right)  ,\ \ j=1,\ldots,k,\ i=1,\ldots,n-1. \label{bp}%
\end{equation}

\end{theorem}

\begin{definition}
The ring obeying the double decomposition of the first kind (\ref{n1}%
)-(\ref{n2}) (of the second kind (\ref{nn1})-(\ref{nn2})) is called
\textsl{polyadic ring of the first kind} (resp. \textsl{of the second kind}).
\end{definition}

\begin{proposition}
The polyadic rings of the first and second kind are not isomorphic.
\end{proposition}

\begin{proof}
This follows from the manifest forms (\ref{n1})-(\ref{n2}) and (\ref{nn1}%
)-(\ref{nn2}). Also, in general case, the $\mathrm{\hat{B}}$-matrices can be
nonsquare (\ref{bp}).
\end{proof}

Thus the two double decompositions introduced above can lead to a new
classification for polyadic analogs of semisimple rings.

\begin{example}
Let us consider the double decomposition of two kinds for ternary ($n=3$)
rings with two semisimple components ($k=2$) and blocks as full $q\times q$
matrix rings over $\mathbb{C}$. Indeed, we have for the ternary nonderived
rings $\mathcal{R}_{semispl}^{\left[  2,3\right]  }$ and $\mathcal{\hat{R}%
}_{semispl}^{\left[  2,3\right]  }$ of the first and second kind,
respectively, the following block structures%
\begin{align}
\mathrm{N}^{shift\text{-}diag\left(  3,2\right)  }\left(  4q\times4q\right)
&  =\left(
\begin{array}
[c]{cccc}%
0 & 0 & A_{1} & 0\\
0 & 0 & 0 & A_{2}\\
B_{1} & 0 & 0 & 0\\
0 & B_{2} & 0 & 0
\end{array}
\right)  ,\ \nonumber\\
\mathrm{\hat{N}}^{diag\text{-}shift\left(  3,2\right)  }\left(  4q\times
4q\right)   &  =\left(
\begin{array}
[c]{cccc}%
0 & \hat{A}_{1} & 0 & 0\\
\hat{A}_{2} & 0 & 0 & 0\\
0 & 0 & 0 & \hat{B}_{1}\\
0 & 0 & \hat{B}_{2} & 0
\end{array}
\right)  , \label{nn}%
\end{align}
where $A_{i},B_{i},\hat{A}_{i},\hat{B}_{i}\in Mat_{q\times q}^{full}\left(
\mathbb{C}\right)  $. In terms of component blocks, the ternary
multiplications in the rings $\mathcal{R}_{semispl}^{\left[  2,3\right]  }$
and $\mathcal{\hat{R}}_{semispl}^{\left[  2,3\right]  }$ are

kind I:%
\begin{align}
A_{1}^{\prime}B_{1}^{\prime\prime}A_{1}^{\prime\prime\prime}  &
=A_{1},\ \ \ A_{2}^{\prime}B_{2}^{\prime\prime}A_{2}^{\prime\prime\prime
}=A_{2},\label{a1}\\
B_{1}^{\prime}A_{1}^{\prime\prime}B_{1}^{\prime\prime\prime}  &
=B_{1},\ \ \ B_{2}^{\prime}A_{2}^{\prime\prime}B_{2}^{\prime\prime\prime
}=B_{2}. \label{a2}%
\end{align}

kind II:%
\begin{align}
\hat{A}_{1}^{\prime}\hat{A}_{2}^{\prime\prime}\hat{A}_{1}^{\prime\prime
\prime}  &  =\hat{A}_{1},\ \ \ \hat{A}_{2}^{\prime}\hat{A}_{1}^{\prime\prime
}\hat{A}_{2}^{\prime\prime\prime}=\hat{A}_{2},\label{aa1}\\
\hat{B}_{1}^{\prime}\hat{B}_{2}^{\prime\prime}\hat{B}_{1}^{\prime\prime
\prime}  &  =\hat{B}_{1},\ \ \ \hat{B}_{2}^{\prime}\hat{B}_{1}^{\prime\prime
}\hat{B}_{2}^{\prime\prime\prime}=\hat{B}_{2}. \label{aa2}%
\end{align}

It follows from (\ref{a1})--(\ref{a2}) and (\ref{aa1})--(\ref{aa2}) that
$\mathcal{R}_{semispl}^{\left[  2,3\right]  }$ and $\mathcal{\hat{R}%
}_{semispl}^{\left[  2,3\right]  }$ are not ternary isomorphic.
\end{example}

Note that the sum of the block structures (\ref{nn}) obeys nontrivial properties.

\begin{remark}
Consider a binary sum of the block matrices of the first and second kind
(\ref{nn})%
\begin{equation}
\mathrm{P}^{\left(  3,2\right)  }\left(  4q\times4q\right)  =\mathrm{N}%
^{shift\text{-}diag\left(  3,2\right)  }\left(  4q\times4q\right)
+\mathrm{\hat{N}}^{diag\text{-}shift\left(  3,2\right)  }\left(
4q\times4q\right)  =\left(
\begin{array}
[c]{cccc}%
0 & \hat{A}_{1} & A_{1} & 0\\
\hat{A}_{2} & 0 & 0 & A_{2}\\
B_{1} & 0 & 0 & \hat{B}_{1}\\
0 & B_{2} & \hat{B}_{2} & 0
\end{array}
\right)  . \label{p}%
\end{equation}
The set of matrices (\ref{p}) forms the nonderived $\left[  2,3\right]  $-ring
$\mathcal{P}^{\left[  2,3\right]  }$ over $\mathbb{C}$%
\begin{equation}
\mathcal{P}^{\left[  2,3\right]  }=\left\langle \left\{  \mathrm{P}^{\left(
3,2\right)  }\left(  4q\times4q\right)  \right\}  \mid\nu^{\left[  2\right]
},\mu^{\left[  3\right]  }\right\rangle ,
\end{equation}
where $\nu^{\left[  2\right]  },\mu^{\left[  3\right]  }$ are binary addition
and ordinary product of $3$ matrices (\ref{p}).

Notice that the $\mathrm{P}$-matrices (\ref{p}) are the block-matrix version
of the circle matrices $\mathsf{M}_{circ}$ which were studied in
\cite{dup/vog2021b} in connection with $8$-vertex solutions to the constant
Yang-Baxter equation \cite{lam/rad} and the corresponding braiding quantum
gates \cite{kau/lom2004,mel/mir/mor}.
\end{remark}

\subsubsection{Supersymmetric double decomposition}

Let us generalize the above double decomposition (of the first kind) to
superrings and superalgebras. For that we first assume that the constituent
vector spaces (entering in (\ref{rh2})) are super vector spaces ($\mathbb{Z}%
_{2}$-graded vector spaces) obeying the standard decomposition into even and
odd parts%
\begin{equation}
W_{i}^{\left(  j\right)  }\left(  q^{\left(  j\right)  }\mid\mathcal{D}%
^{\left(  j\right)  }\right)  =W_{i}^{\left(  j\right)  }\left(
q_{even}^{\left(  j\right)  }\mid\mathcal{D}^{\left(  j\right)  }\right)
_{even}\oplus W_{i}^{\left(  j\right)  }\left(  q_{odd}^{\left(  j\right)
}\mid\mathcal{D}^{\left(  j\right)  }\right)  _{odd},\ i=1,\ldots
,n-1,\ j=1,\ldots,k, \label{ws}%
\end{equation}
where $q_{even}^{\left(  j\right)  }$ and $q_{odd}^{\left(  j\right)  }$ are
dimensions of the even and odd spaces,respectively, $q^{\left(  j\right)
}=q_{even}^{\left(  j\right)  }+q_{odd}^{\left(  j\right)  }$.

The parity of a homogeneous element of the vector space $v\in W_{i}^{\left(
j\right)  }\left(  q^{\left(  j\right)  }\mid\mathcal{D}^{\left(  j\right)
}\right)  $ is defined by $\left\vert v\right\vert =\bar{0}$ (resp. $\bar{1}%
$), if $v\in W_{i}^{\left(  j\right)  }\left(  q_{even}^{\left(  j\right)
}\mid\mathcal{D}^{\left(  j\right)  }\right)  _{even}$ (resp. $W_{i}^{\left(
j\right)  }\left(  q_{odd}^{\left(  j\right)  }\mid\mathcal{D}^{\left(
j\right)  }\right)  _{odd}$), and $\bar{0},\bar{1}\in\mathbb{Z}_{2}$. For
details, see \cite{berezin,leites}. In the graded case, the $k$ square blocks
$\mathrm{A}$'s in (\ref{aqq}) are full supermatrix rings of the size $\left(
q_{even}^{\left(  j\right)  }\mid q_{odd}^{\left(  j\right)  }\right)
\times\left(  q_{even}^{\left(  j\right)  }\mid q_{odd}^{\left(  j\right)
}\right)  $, while the square $\mathrm{B}$'s (\ref{n2}) are block-diagonal
supermatrices, and the block-shift $n$-ary supermatrices have a nonstandard
form (\ref{n1}).

We assume that in super case a polyadic analog of the Wedderburn-Artin theorem
for semisimple Artinian superrings (of the first kind) is also valid, with the
form of the double decomposition (\ref{n1})--(\ref{n2}) being the same,
however now the blocks $\mathrm{A}$'s and $\mathrm{B}$'s are corresponding supermatrices.

\section{\textsc{Polyadization concept}}

Here we propose a general procedure for how to construct new polyadic
algebraic structures from binary (or lower arity) ones, using the
\textquotedblleft inverse\textquotedblright\ (informally) to the block-shift
matrix decomposition (\ref{mb}). It can be considered as a polyadic analog of
the inverse problem of the determination of an algebraic structure from the
knowledge of its Wedderburn decomposition \cite{die/mit}.

\subsection{Polyadization of binary algebraic structures}

Let a binary algebraic structure $\mathcal{X}$ be represented by $p\times p$
matrices $\mathrm{B}_{\mathbf{y}}\equiv\mathrm{B}_{\mathbf{y}}\left(  p\times
p\right)  $ over a ring $\mathcal{R}$ (a linear representation), where
$\mathbf{y}$ is the set of $N_{y}$ parameters corresponding to an element $x$
of $\mathcal{X}$. Because the binary addition in $\mathcal{R}$ transfers to
the matrix addition without restrictions (as opposed to the polyadic case, see
below), we will consider only the multiplicative part of the resulting
polyadic matrix ring. In this way, we propose a special block-shift matrix
method to obtain $n$-ary semigroups ($n$-ary groups) from the binary ones, but
the former are not derived from the latter \cite{galmak1,duplij2022}. In
general, this can lead to new algebraic structures that were not known before.

\begin{definition}
\label{def-pol}A (\textit{block-matrix})\textit{ polyadization} $\mathbf{\Phi
}_{pol}$ of a binary semigroup (or group) $\mathcal{X}$ represented by square
$p\times p$ matrices $\mathrm{B}_{\mathbf{y}}$ is an $n$-ary semigroup (or an
$n$-ary group) represented by the $d\times d$ block-shift matrices (over a
ring $\mathcal{R}$) of the form (\ref{mb}) as follows%
\begin{equation}
\mathrm{Q}_{\mathbf{y}_{1},\ldots,\mathbf{y}_{n-1}}\equiv\mathrm{Q}%
_{\mathbf{y}_{1},\ldots,\mathbf{y}_{n-1}}^{Bshift\left(  n\right)  }\left(
d\times d\right)  =\left(
\begin{array}
[c]{ccccc}%
0 & \mathrm{B}_{\mathbf{y}_{1}} & \ldots & 0 & 0\\
0 & 0 & \mathrm{B}_{\mathbf{y}_{2}} & \ldots & 0\\
0 & 0 & \ddots & \ddots & \vdots\\
\vdots & \vdots & \ddots & 0 & \mathrm{B}_{\mathbf{y}_{n-2}}\\
\mathrm{B}_{\mathbf{y}_{n-1}} & 0 & \ldots & 0 & 0
\end{array}
\right)  , \label{q}%
\end{equation}
where $d=\left(  n-1\right)  p$, and the $n$-ary multiplication $\mu^{\left[
\left[  n\right]  \right]  }$ is given by the product of $n$ matrices (\ref{q}).
\end{definition}

In terms of the block-matrices $\mathrm{B}$'s the multiplication%
\begin{equation}
\mu^{\left[  \left[  n\right]  \right]  }\left[  \overset{n}{\overbrace
{\mathrm{Q}_{\mathbf{y}_{1}^{\prime},\ldots,\mathbf{y}_{n-1}^{\prime}%
},\mathrm{Q}_{\mathbf{y}_{1}^{\prime\prime},\ldots,\mathbf{y}_{n-1}%
^{\prime\prime}},\ldots,\mathrm{Q}_{\mathbf{y}_{1}^{\prime\prime\prime}%
,\ldots,\mathbf{y}_{n-1}^{\prime\prime\prime}}\mathrm{Q}_{\mathbf{y}%
_{1}^{\prime\prime\prime\prime},\ldots,\mathbf{y}_{n-1}^{\prime\prime
\prime\prime}}}}\right]  =\mathrm{Q}_{\mathbf{y}_{1},\ldots,\mathbf{y}_{n-1}}
\label{mq}%
\end{equation}
has the cyclic product form (see \cite{dup2021b})%
\begin{align}
\overset{n}{\overbrace{\mathrm{B}_{\mathbf{y}_{1}^{\prime}}\mathrm{B}%
_{\mathbf{y}_{2}^{\prime\prime}}\ldots\mathrm{B}_{\mathbf{y}_{n-1}%
^{\prime\prime\prime}}\mathrm{B}_{\mathbf{y}_{1}^{\prime\prime\prime\prime}}%
}}  &  =\mathrm{B}_{\mathbf{y}_{1}},\label{b1}\\
\mathrm{B}_{\mathbf{y}_{2}^{\prime}}\mathrm{B}_{\mathbf{y}_{3}^{\prime\prime}%
}\ldots\mathrm{B}_{\mathbf{y}_{1}^{\prime\prime\prime}}\mathrm{B}%
_{\mathbf{y}_{2}^{\prime\prime\prime\prime}}  &  =\mathrm{B}_{\mathbf{y}_{2}%
},\\
&  \vdots\nonumber\\
\mathrm{B}_{\mathbf{y}_{n-1}^{\prime}}\mathrm{B}_{\mathbf{y}_{1}^{\prime
\prime}}\ldots\mathrm{B}_{\mathbf{y}_{n-2}^{\prime\prime\prime}}%
\mathrm{B}_{\mathbf{y}_{n-1}^{\prime\prime\prime\prime}}  &  =\mathrm{B}%
_{\mathbf{y}_{n-1}}. \label{b2}%
\end{align}

\begin{remark}
The number of parameters $N_{y}$ describing an element $x\in\mathcal{X}$
increases to $\left(  n-1\right)  N_{y}$, and the corresponding algebraic
structure $\left\langle \left\{  \mathrm{Q}_{\mathbf{y}_{1},\ldots
,\mathbf{y}_{n-1}}\right\}  \mid\mu^{\left[  \left[  n\right]  \right]
}\right\rangle $ becomes $n$-ary, and so (\ref{q}) can be treated as a new
algebraic structure, which we denote by the same letter with the arities in
double square brackets $\mathcal{X}^{\left[  \left[  n\right]  \right]  }$.
\end{remark}

We now analyze some of the most general properties of the polyadization map
$\mathbf{\Phi}_{pol}$ which are independent of the concrete form of the
block-matrices $\mathrm{B}$'s and over which algebraic structure (ring, field,
etc...) they are defined. We then present some concrete examples.

\begin{definition}
A \textit{unique polyadization} $\mathbf{\Phi}_{Upol}$ is a polyadization
where all sets of parameters coincide%
\begin{equation}
\mathbf{y}=\mathbf{y}_{1}=\mathbf{y}_{2}\ldots=\mathbf{y}_{n-1}. \label{y}%
\end{equation}

\end{definition}

\begin{proposition}
The unique polyadization is an $n$-ary-binary homomorphism.
\end{proposition}

\begin{proof}
In the case of (\ref{y}) all $\left(  n-1\right)  $ relations (\ref{b1}%
)--(\ref{b2}) coincide%
\begin{equation}
\overset{n}{\overbrace{\mathrm{B}_{\mathbf{y}^{\prime}}\mathrm{B}%
_{\mathbf{y}^{\prime\prime}}\ldots\mathrm{B}_{\mathbf{y}^{\prime\prime\prime}%
}\mathrm{B}_{\mathbf{y}^{\prime\prime\prime\prime}}}}=\mathrm{B}_{\mathbf{y}},
\label{bb}%
\end{equation}
which means that the ordinary (binary) product of $n$ matrices $\mathrm{B}%
_{\mathbf{y}}$'s is mapped to the $n$-ary product of matrices $\mathrm{Q}%
_{\mathbf{y}}$'s (\ref{mq})%
\begin{equation}
\mu^{\left[  \left[  n\right]  \right]  }\left[  \overset{n}{\overbrace
{\mathrm{Q}_{\mathbf{y}^{\prime}},\mathrm{Q}_{\mathbf{y}^{\prime\prime}%
},\ldots,\mathrm{Q}_{\mathbf{y}^{\prime\prime\prime}}\mathrm{Q}_{\mathbf{y}%
^{\prime\prime\prime\prime}}}}\right]  =\mathrm{Q}_{\mathbf{y}}, \label{qq}%
\end{equation}
as it should be for an $n$-ary-binary homomorphism, but not for a homomorphism.
\end{proof}

\begin{assertion}
If matrices $\mathrm{B}_{\mathbf{y}}\equiv\mathrm{B}_{\mathbf{y}}\left(
p\times p\right)  $ contain the identity matrix $\mathrm{E}_{p}$, then the
$n$-ary identity $\mathrm{E}_{d}^{\left(  n\right)  }$ in $\left\langle
\left\{  \mathrm{Q}_{\mathbf{y}}\left(  d\times d\right)  \right\}  \mid
\mu^{\left[  \left[  n\right]  \right]  }\right\rangle $, $d=\left(
n-1\right)  p$ has the form%
\begin{equation}
\mathrm{E}_{d}^{\left(  n\right)  }=\left(
\begin{array}
[c]{ccccc}%
0 & \mathrm{E}_{p} & \ldots & 0 & 0\\
0 & 0 & \mathrm{E}_{p} & \ldots & 0\\
0 & 0 & \ddots & \ddots & \vdots\\
\vdots & \vdots & \ddots & 0 & \mathrm{E}_{p}\\
\mathrm{E}_{p} & 0 & \ldots & 0 & 0
\end{array}
\right)  . \label{e}%
\end{equation}

\end{assertion}

\begin{proof}
It follows from (\ref{q}), (\ref{mq}) and (\ref{bb}).
\end{proof}

In this case the unique polyadization maps the identity matrix to the $n$-ary
identity $\mathbf{\Phi}_{Upol}:\mathrm{E}_{p}\rightarrow\mathrm{E}%
_{d}^{\left(  n\right)  }$.

\begin{assertion}
If the matrices $\mathrm{B}_{\mathbf{y}}$ are invertible $\mathrm{B}%
_{\mathbf{y}}\mathrm{B}_{\mathbf{y}}^{-1}=\mathrm{B}_{\mathbf{y}}%
^{-1}\mathrm{B}_{\mathbf{y}}=\mathrm{E}_{p}$, then each $\mathrm{Q}%
_{\mathbf{y}_{1},\ldots,\mathbf{y}_{n-1}}$ has a querelement%
\begin{equation}
\overline{\mathrm{Q}}_{\mathbf{y}_{1},\ldots,\mathbf{y}_{n-1}}=\left(
\begin{array}
[c]{ccccc}%
0 & \overline{\mathrm{B}}_{\mathbf{y}_{1}} & \ldots & 0 & 0\\
0 & 0 & \overline{\mathrm{B}}_{\mathbf{y}_{2}} & \ldots & 0\\
0 & 0 & \ddots & \ddots & \vdots\\
\vdots & \vdots & \ddots & 0 & \overline{\mathrm{B}}_{\mathbf{y}_{n-2}}\\
\overline{\mathrm{B}}_{\mathbf{y}_{n-1}} & 0 & \ldots & 0 & 0
\end{array}
\right)  , \label{qb}%
\end{equation}
satisfying%
\begin{equation}
\mu^{\left[  \left[  n\right]  \right]  }\left[  \overset{n}{\overbrace
{\mathrm{Q}_{\mathbf{y}_{1},\ldots,\mathbf{y}_{n-1}},\mathrm{Q}_{\mathbf{y}%
_{1},\ldots,\mathbf{y}_{n-1}},\ldots,\mathrm{Q}_{\mathbf{y}_{1},\ldots
,\mathbf{y}_{n-1}}\overline{\mathrm{Q}}_{\mathbf{y}_{1},\ldots,\mathbf{y}%
_{n-1}}}}\right]  =\mathrm{Q}_{\mathbf{y}_{1},\ldots,\mathbf{y}_{n-1}}
\label{mqq}%
\end{equation}
where $\overline{\mathrm{Q}}_{\mathbf{y}_{1},\ldots,\mathbf{y}_{n-1}}$ can be
on any places and%
\begin{equation}
\overline{\mathrm{B}}_{\mathbf{y}_{i}}=\mathrm{B}_{\mathbf{y}_{i-1}}%
^{-1}\mathrm{B}_{\mathbf{y}_{i-2}}^{-1}\ldots\mathrm{B}_{\mathbf{y}_{2}}%
^{-1}\mathrm{B}_{\mathbf{y}_{1}}^{-1}\mathrm{B}_{\mathbf{y}_{n-1}}%
^{-1}\mathrm{B}_{\mathbf{y}_{n-2}}^{-1}\ldots\mathrm{B}_{\mathbf{y}_{i+2}%
}^{-1}\mathrm{B}_{\mathbf{y}_{i+1}}^{-1}. \label{bb1}%
\end{equation}

\end{assertion}

\begin{proof}
This follows from (\ref{qb})--(\ref{mqq}) and (\ref{b1})--(\ref{b2}), then
consequently applying $\mathrm{B}_{\mathbf{y}_{i}}^{-1}$ (with suitable
indices) on both sides, we obtain (\ref{bb1}).
\end{proof}

Let us suppose that on the set of matrices $\left\{  \mathrm{B}_{\mathbf{y}%
}\right\}  $ over a binary ring $\mathcal{R}$ one can consider some analog of
a multiplicative character $\chi:\left\{  \mathrm{B}_{\mathbf{y}}\right\}
\rightarrow\mathcal{R}$, being a (binary) homomorphism, such that%
\begin{equation}
\chi\left(  \mathrm{B}_{\mathbf{y}_{1}}\right)  \chi\left(  \mathrm{B}%
_{\mathbf{y}_{2}}\right)  =\chi\left(  \mathrm{B}_{\mathbf{y}_{1}}%
\mathrm{B}_{\mathbf{y}_{2}}\right)  . \label{xb}%
\end{equation}

For instance, in case $\mathrm{B}\in GL\left(  p,\mathbb{C}\right)  $, the
determinant can be considered as a (binary) multiplicative character.
Similarly, we can introduce

\begin{definition}
A \textit{polyadized multiplicative character} $\mathbf{\chi}:\left\{
\mathrm{Q}_{\mathbf{y}_{1},\ldots,\mathbf{y}_{n-1}}\right\}  \rightarrow
\mathcal{R}$ is proportional to a product of the binary multiplicative
characters of the blocks $\chi\left(  \mathrm{B}_{\mathbf{y}_{i}}\right)  $%
\begin{equation}
\mathbf{\chi}\left(  \mathrm{Q}_{\mathbf{y}_{1},\ldots,\mathbf{y}_{n-1}%
}\right)  =\left(  -1\right)  ^{n}\chi\left(  \mathrm{B}_{\mathbf{y}_{1}%
}\right)  \chi\left(  \mathrm{B}_{\mathbf{y}_{2}}\right)  \ldots\chi\left(
\mathrm{B}_{\mathbf{y}_{n-1}}\right)  . \label{xq}%
\end{equation}

\end{definition}

The normalization factor $\left(  -1\right)  ^{n}$ in (\ref{xq}) is needed to
be consistent with the case when $\mathcal{R}$ is commutative, and the
multiplicative characters are determinants. It can also be consistent in other cases.

\begin{proposition}
If the ring $\mathcal{R}$ is commutative, then the polyadized multiplicative
character $\mathbf{\chi}$ is an $n$-ary-binary homomorphism.
\end{proposition}

\begin{proof}
It follows from (\ref{bb})--(\ref{qq}), (\ref{xq}) and the commutativity of
$\mathcal{R}$.
\end{proof}

\begin{proposition}
If the ring $\mathcal{R}$ is commutative and unital with the unit $E_{p}$,
then the algebraic structure $\left\langle \left\{  \mathrm{Q}_{\mathbf{y}%
_{1},\ldots,\mathbf{y}_{n-1}}\right\}  \mid\mu^{\left[  \left[  n\right]
\right]  }\right\rangle $ contains polyadic ($n$-ary) idempotents satisfying%
\begin{equation}
\mathrm{B}_{\mathbf{y}_{1}}\mathrm{B}_{\mathbf{y}_{2}}\ldots\mathrm{B}%
_{\mathbf{y}_{n-1}}=E_{p}. \label{be}%
\end{equation}

\end{proposition}

\begin{proof}
It follows from (\ref{qq}) and (\ref{e}).
\end{proof}

\subsection{Concrete examples of the polyadization procedure}

\subsubsection{Polyadization of $GL\left(  2,\mathbb{C}\right)  $}

Consider the polyadization procedure for the general linear group $GL\left(
2,\mathbb{C}\right)  $. We have for the $4$-parameter block matrices
$\mathrm{B}_{\mathbf{y}_{i}}=\left(
\begin{array}
[c]{cc}%
a_{i} & b_{i}\\
c_{i} & d_{i}%
\end{array}
\right)  \in GL\left(  2,\mathbb{C}\right)  $, $\mathbf{y}_{i}=\left(
a_{i},b_{i},c_{i},d_{i}\right)  \in\mathbb{C\times}\mathbb{C\times C\times
C},$ $i=1,2,3$. Thus, the $12$-parameter $4$-ary group $GL^{\left[  \left[
4\right]  \right]  }\left(  2,\mathbb{C}\right)  =\left\langle \left\{
\mathrm{Q}_{\mathbf{y}_{1},\mathbf{y}_{2},\mathbf{y}_{3}}\right\}  \mid
\mu^{\left[  \left[  4\right]  \right]  }\right\rangle $ is represented by the
following $6\times6$ $\mathrm{Q}$-matrices%
\begin{equation}
\mathrm{Q}_{\mathbf{y}_{1},\mathbf{y}_{2},\mathbf{y}_{3}}=\left(
\begin{array}
[c]{ccc}%
0 & \mathrm{B}_{\mathbf{y}_{1}} & 0\\
0 & 0 & \mathrm{B}_{\mathbf{y}_{2}}\\
\mathrm{B}_{\mathbf{y}_{3}} & 0 & 0
\end{array}
\right)  \in GL^{\left[  \left[  4\right]  \right]  }\left(  2,\mathbb{C}%
\right)  ,\ \ \ \mathrm{B}_{\mathbf{y}_{i}}\in GL\left(  2,\mathbb{C}\right)
,\ \ \ \ i=1,2,3, \label{q3}%
\end{equation}
obeying the $4$-ary multiplication%
\begin{equation}
\mu^{\left[  \left[  4\right]  \right]  }\left[  \mathrm{Q}_{\mathbf{y}%
_{1}^{\prime},\mathbf{y}_{2}^{\prime},\mathbf{y}_{3}^{\prime}},\mathrm{Q}%
_{\mathbf{y}_{1}^{\prime\prime},\mathbf{y}_{2}^{\prime\prime},\mathbf{y}%
_{3}^{\prime\prime}},\mathrm{Q}_{\mathbf{y}_{1}^{\prime\prime\prime
},\mathbf{y}_{2}^{\prime\prime\prime},\mathbf{y}_{3}^{\prime\prime\prime}%
},\mathrm{Q}_{\mathbf{y}_{1}^{\prime\prime\prime\prime},\mathbf{y}_{2}%
^{\prime\prime\prime\prime},\mathbf{y}_{3}^{\prime\prime\prime\prime}}\right]
=\mathrm{Q}_{\mathbf{y}_{1}^{\prime},\mathbf{y}_{2}^{\prime},\mathbf{y}%
_{3}^{\prime}}\mathrm{Q}_{\mathbf{y}_{1}^{\prime\prime},\mathbf{y}_{2}%
^{\prime\prime},\mathbf{y}_{3}^{\prime\prime}}\mathrm{Q}_{\mathbf{y}%
_{1}^{\prime\prime\prime},\mathbf{y}_{2}^{\prime\prime\prime},\mathbf{y}%
_{3}^{\prime\prime\prime}}\mathrm{Q}_{\mathbf{y}_{1}^{\prime\prime\prime
\prime},\mathbf{y}_{2}^{\prime\prime\prime\prime},\mathbf{y}_{3}^{\prime
\prime\prime\prime}}=\mathrm{Q}_{\mathbf{y}_{1},\mathbf{y}_{2},\mathbf{y}_{3}%
}. \label{mq4}%
\end{equation}
In terms of the block matrices $\mathrm{B}_{\mathbf{y}_{i}}$ the
multiplication (\ref{mq4}) becomes (see (\ref{mq})--(\ref{b2}))%
\begin{align}
\mathrm{B}_{\mathbf{y}_{1}^{\prime}}\mathrm{B}_{\mathbf{y}_{2}^{\prime\prime}%
}\mathrm{B}_{\mathbf{y}_{3}^{\prime\prime\prime}}\mathrm{B}_{\mathbf{y}%
_{1}^{\prime\prime\prime\prime}}  &  =\mathrm{B}_{\mathbf{y}_{1}}%
,\label{bbb1}\\
\mathrm{B}_{\mathbf{y}_{2}^{\prime}}\mathrm{B}_{\mathbf{y}_{3}^{\prime\prime}%
}\mathrm{B}_{\mathbf{y}_{1}^{\prime\prime\prime}}\mathrm{B}_{\mathbf{y}%
_{2}^{\prime\prime\prime\prime}}  &  =\mathrm{B}_{\mathbf{y}_{2}},\\
\mathrm{B}_{\mathbf{y}_{3}^{\prime}}\mathrm{B}_{\mathbf{y}_{1}^{\prime\prime}%
}\mathrm{B}_{\mathbf{y}_{2}^{\prime\prime\prime}}\mathrm{B}_{\mathbf{y}%
_{3}^{\prime\prime\prime\prime}}  &  =\mathrm{B}_{\mathbf{y}_{3}},
\label{bbb3}%
\end{align}
which can be further expressed in the $\mathrm{B}$-matrix entries (its
manifest form is too cumbersome to give here).

For $\left\{  \mathrm{Q}_{\mathbf{y}_{1},\mathbf{y}_{2},\mathbf{y}_{3}%
}\right\}  $ to be a $4$-ary group each $\mathrm{Q}$-matrix should have the
unique querelement determined by the equation (see (\ref{mqq}))%
\begin{equation}
\mathrm{Q}_{\mathbf{y}_{1},\mathbf{y}_{2},\mathbf{y}_{3}}\mathrm{Q}%
_{\mathbf{y}_{1},\mathbf{y}_{2},\mathbf{y}_{3}}\mathrm{Q}_{\mathbf{y}%
_{1},\mathbf{y}_{2},\mathbf{y}_{3}}\overline{\mathrm{Q}}_{\mathbf{y}%
_{1},\mathbf{y}_{2},\mathbf{y}_{3}}=\mathrm{Q}_{\mathbf{y}_{1},\mathbf{y}%
_{2},\mathbf{y}_{3}}, \label{qqq}%
\end{equation}
which has the solution%
\begin{equation}
\overline{\mathrm{Q}}_{\mathbf{y}_{1},\mathbf{y}_{2},\mathbf{y}_{3}}=\left(
\begin{array}
[c]{ccc}%
0 & \overline{\mathrm{B}}_{\mathbf{y}_{1}} & 0\\
0 & 0 & \overline{\mathrm{B}}_{\mathbf{y}_{2}}\\
\overline{\mathrm{B}}_{\mathbf{y}_{3}} & 0 & 0
\end{array}
\right)  , \label{qbb}%
\end{equation}
where (see (\ref{bb1}))%
\begin{equation}
\overline{\mathrm{B}}_{\mathbf{y}_{1}}=\mathrm{B}_{\mathbf{y}_{3}}%
^{-1}\mathrm{B}_{\mathbf{y}_{2}}^{-1},\ \ \ \ \overline{\mathrm{B}%
}_{\mathbf{y}_{2}}=\mathrm{B}_{\mathbf{y}_{1}}^{-1}\mathrm{B}_{\mathbf{y}_{3}%
}^{-1},\ \ \ \ \overline{\mathrm{B}}_{\mathbf{y}_{3}}=\mathrm{B}%
_{\mathbf{y}_{2}}^{-1}\mathrm{B}_{\mathbf{y}_{1}}^{-1}.
\end{equation}

In the manifest form the querelements of $GL^{\left[  \left[  4\right]
\right]  }\left(  2,\mathbb{C}\right)  $ are (\ref{qbb}), where
\begin{align}
\overline{\mathrm{B}}_{\mathbf{y}_{1}}  &  =\frac{1}{\Delta_{3}\Delta_{2}%
}\left(
\begin{array}
[c]{cc}%
b_{3}c_{2}+d_{3}d_{2} & -b_{3}a_{2}-d_{3}b_{2}\\
-a_{3}c_{2}-c_{3}d_{2} & a_{3}a_{2}+c_{3}b_{2}%
\end{array}
\right) \\
\overline{\mathrm{B}}_{\mathbf{y}_{2}}  &  =\frac{1}{\Delta_{2}\Delta_{3}%
}\left(
\begin{array}
[c]{cc}%
b_{1}c_{3}+d_{1}d_{3} & -b_{1}a_{3}-d_{1}b_{3}\\
-a_{1}c_{3}-c_{1}d_{3} & a_{1}a_{3}+c_{1}b_{3}%
\end{array}
\right) \\
\overline{\mathrm{B}}_{\mathbf{y}_{3}}  &  =\frac{1}{\Delta_{2}\Delta_{1}%
}\left(
\begin{array}
[c]{cc}%
b_{2}c_{1}+d_{2}d_{1} & -b_{2}a_{1}-d_{2}b_{1}\\
-a_{2}c_{1}-c_{2}d_{1} & a_{2}a_{1}+c_{2}b_{1}%
\end{array}
\right)  ,
\end{align}
where $\Delta_{i}=a_{i}d_{i}-b_{i}c_{i}\neq0$ are the (nonvanishing)
determinants of $\mathrm{B}_{\mathbf{y}_{i}}$.

\begin{definition}
We call $GL^{\left[  \left[  4\right]  \right]  }\left(  2,\mathbb{C}\right)
$ a polyadic ($4$-ary) general linear group.
\end{definition}

If we take the binary multiplicative characters to be determinants
$\chi\left(  \mathrm{B}_{\mathbf{y}_{i}}\right)  =\Delta_{i}\neq0$, then the
polyadized multiplicative character in $GL^{\left[  \left[  4\right]  \right]
}\left(  2,\mathbb{C}\right)  $ becomes%
\begin{equation}
\mathbf{\chi}\left(  \mathrm{Q}_{\mathbf{y}_{1},\mathbf{y}_{2},\mathbf{y}_{3}%
}\right)  =\Delta_{1}\Delta_{2}\Delta_{3},
\end{equation}
which is a $4$-ary-binary homomorphism, because (see (\ref{bbb1}%
)--(\ref{bbb3}))%
\begin{align}
&  \mathbf{\chi}\left(  \mathrm{Q}_{\mathbf{y}_{1}^{\prime},\mathbf{y}%
_{2}^{\prime},\mathbf{y}_{3}^{\prime}}\right)  \mathbf{\chi}\left(
\mathrm{Q}_{\mathbf{y}_{1}^{\prime\prime},\mathbf{y}_{2}^{\prime\prime
},\mathbf{y}_{3}^{\prime\prime}}\right)  \mathbf{\chi}\left(  \mathrm{Q}%
_{\mathbf{y}_{1}^{\prime\prime\prime},\mathbf{y}_{2}^{\prime\prime\prime
},\mathbf{y}_{3}^{\prime\prime\prime}}\right)  \mathbf{\chi}\left(
\mathrm{Q}_{\mathbf{y}_{1}^{\prime\prime\prime},\mathbf{y}_{2}^{\prime
\prime\prime},\mathbf{y}_{3}^{\prime\prime\prime}}\right) \nonumber\\
&  =\left(  \Delta_{1}^{\prime}\Delta_{2}^{\prime}\Delta_{3}^{\prime}\right)
\left(  \Delta_{1}^{\prime\prime}\Delta_{2}^{\prime\prime}\Delta_{3}%
^{\prime\prime}\right)  \left(  \Delta_{1}^{\prime\prime\prime}\Delta
_{2}^{\prime\prime\prime}\Delta_{3}^{\prime\prime\prime}\right)  \left(
\Delta_{1}^{\prime\prime\prime\prime}\Delta_{2}^{\prime\prime\prime\prime
}\Delta_{3}^{\prime\prime\prime\prime}\right) \nonumber\\
&  =\left(  \Delta_{1}^{\prime}\Delta_{2}^{\prime\prime}\Delta_{3}%
^{\prime\prime\prime}\Delta_{1}^{\prime\prime\prime\prime}\right)  \left(
\Delta_{2}^{\prime}\Delta_{3}^{\prime\prime}\Delta_{1}^{\prime\prime\prime
}\Delta_{2}^{\prime\prime\prime\prime}\right)  \left(  \Delta_{3}^{\prime
}\Delta_{1}^{\prime\prime}\Delta_{2}^{\prime\prime\prime}\Delta_{3}%
^{\prime\prime\prime\prime}\right) \nonumber\\
&  =\mathbf{\chi}\left(  \mathrm{Q}_{\mathbf{y}_{1}^{\prime},\mathbf{y}%
_{2}^{\prime},\mathbf{y}_{3}^{\prime}}\mathrm{Q}_{\mathbf{y}_{1}^{\prime
\prime},\mathbf{y}_{2}^{\prime\prime},\mathbf{y}_{3}^{\prime\prime}}%
\mathrm{Q}_{\mathbf{y}_{1}^{\prime\prime\prime},\mathbf{y}_{2}^{\prime
\prime\prime},\mathbf{y}_{3}^{\prime\prime\prime}}\mathrm{Q}_{\mathbf{y}%
_{1}^{\prime\prime\prime\prime},\mathbf{y}_{2}^{\prime\prime\prime\prime
},\mathbf{y}_{3}^{\prime\prime\prime\prime}}\right)  . \label{xqq}%
\end{align}

The $4$-ary identity $\mathrm{E}_{6}^{\left(  4\right)  }$ of $GL^{\left[
\left[  4\right]  \right]  }\left(  2,\mathbb{C}\right)  $ is unique and has
the form (see (\ref{e}))%
\begin{equation}
\mathrm{E}_{6}^{\left(  4\right)  }=\left(
\begin{array}
[c]{ccc}%
0 & \mathrm{E}_{2} & 0\\
0 & 0 & \mathrm{E}_{2}\\
\mathrm{E}_{2} & 0 & 0
\end{array}
\right)  , \label{e4}%
\end{equation}
where $E_{2}$ is the identity of $GL\left(  2,\mathbb{C}\right)  $. The
$4$-ary identity $\mathrm{E}_{6}^{\left(  4\right)  }$ satisfies the $4$-ary
idempotence relation%
\begin{equation}
\mathrm{E}_{6}^{\left(  4\right)  }\mathrm{E}_{6}^{\left(  4\right)
}\mathrm{E}_{6}^{\left(  4\right)  }\mathrm{E}_{6}^{\left(  4\right)
}=\mathrm{E}_{6}^{\left(  4\right)  }. \label{ee}%
\end{equation}

In general, the $4$-ary group $GL^{\left[  \left[  4\right]  \right]  }\left(
2,\mathbb{C}\right)  $ contains an infinite number of $4$-ary idempotents
$\mathrm{Q}_{\mathbf{y}_{1},\mathbf{y}_{2},\mathbf{y}_{3}}^{idemp}$ defined by
the system of equations%
\begin{equation}
\mathrm{Q}_{\mathbf{y}_{1},\mathbf{y}_{2},\mathbf{y}_{3}}^{idemp}%
\mathrm{Q}_{\mathbf{y}_{1},\mathbf{y}_{2},\mathbf{y}_{3}}^{idemp}%
\mathrm{Q}_{\mathbf{y}_{1},\mathbf{y}_{2},\mathbf{y}_{3}}^{idemp}%
\mathrm{Q}_{\mathbf{y}_{1},\mathbf{y}_{2},\mathbf{y}_{3}}^{idemp}%
=\mathrm{Q}_{\mathbf{y}_{1},\mathbf{y}_{2},\mathbf{y}_{3}}^{idemp},
\end{equation}
which gives%
\begin{equation}
\mathrm{B}_{\mathbf{y}_{1}}^{idemp}\mathrm{B}_{\mathbf{y}_{2}}^{idemp}%
\mathrm{B}_{\mathbf{y}_{3}}^{idemp}=\mathrm{E}_{2}, \label{bbbe}%
\end{equation}
or manifestly%
\begin{align}
a_{1}a_{2}a_{3}+a_{1}b_{2}c_{3}+a_{3}b_{1}c_{2}+b_{1}c_{3}d_{2}  &
=1,\label{ab1}\\
a_{2}b_{3}c_{1}+b_{2}c_{1}d_{3}+b_{3}c_{2}d_{1}+d_{1}d_{2}d_{3}  &  =1,\\
a_{1}a_{2}b_{3}+a_{1}b_{2}d_{3}+b_{1}b_{3}c_{2}+b_{1}d_{2}d_{3}  &  =0,\\
a_{2}a_{3}c_{1}+a_{3}c_{2}d_{1}+b_{2}c_{1}c_{3}+c_{3}d_{1}d_{2}  &  =0.
\label{ab2}%
\end{align}

The infinite set of idempotents in $GL^{\left[  \left[  4\right]  \right]
}\left(  2,\mathbb{C}\right)  $ is determined by $12-4=8$ complex parameters,
because one block-matrix (with $4$ complex parameters) can always be excluded
using the equation (\ref{bbbe}).

\begin{remark}
\label{rem-idemp}The above example shows, how \textquotedblleft
far\textquotedblright\ polyadic groups can be formed from ordinary (binary)
groups: the former can contain infinite number of $4$-ary idempotents
determined by (\ref{ab1})--(\ref{ab2}), in addition to the standard idempotent
in any group, the $4$-ary identity (\ref{e4}).
\end{remark}

\subsubsection{Polyadization of $SO\left(  2,\mathbb{R}\right)  $}

Here we provide a polyadization for the simplest subgroup of $GL\left(
2,\mathbb{C}\right)  $, the special orthogonal group $SO\left(  2,\mathbb{R}%
\right)  $. In the matrix form $SO\left(  2,\mathbb{R}\right)  $ is
represented by the one-parameter rotation matrix%
\begin{equation}
B\left(  \alpha\right)  =\left(
\begin{array}
[c]{cc}%
\cos\alpha & -\sin\alpha\\
\sin\alpha & \cos\alpha
\end{array}
\right)  \in SO\left(  2,\mathbb{R}\right)  ,\ \ \ \ \ \alpha\in
\mathbb{R}\diagup2\pi\mathbb{Z}, \label{ba}%
\end{equation}
satisfying the commutative multiplication%
\begin{equation}
B\left(  \alpha\right)  B\left(  \beta\right)  =B\left(  \alpha+\beta\right)
, \label{bba}%
\end{equation}
and the (binary) identity $\mathrm{E}_{2}$ is $B\left(  0\right)  $.
Therefore, the inverse element for $B\left(  \alpha\right)  $ is $B\left(
-\alpha\right)  $.

The $4$-ary polyadization of $SO\left(  2,\mathbb{R}\right)  $ is given by the
$3$-parameter $4$-ary group of $Q$-matrices $SO^{\left[  \left[  4\right]
\right]  }\left(  2,\mathbb{R}\right)  =\left\langle \left\{  Q\left(
\alpha,\beta,\gamma\right)  \right\}  \mid\mu^{\left[  \left[  4\right]
\right]  }\right\rangle $, where (cf. (\ref{q3}))%
\begin{align}
&  Q\left(  \alpha,\beta,\gamma\right)  =\left(
\begin{array}
[c]{ccc}%
0 & B\left(  \alpha\right)  & 0\\
0 & 0 & B\left(  \beta\right) \\
B\left(  \gamma\right)  & 0 & 0
\end{array}
\right) \label{qa}\\
&  =\left(
\begin{array}
[c]{cccccc}%
0 & 0 & \cos\alpha & -\sin\alpha & 0 & 0\\
0 & 0 & \sin\alpha & \cos\alpha & 0 & 0\\
0 & 0 & 0 & 0 & \cos\beta & -\sin\beta\\
0 & 0 & 0 & 0 & \sin\beta & \cos\beta\\
\cos\gamma & -\sin\gamma & 0 & 0 & 0 & 0\\
\sin\gamma & \cos\gamma & 0 & 0 & 0 & 0
\end{array}
\right)  ,\ \ \ \ \alpha,\beta,\gamma\in\mathbb{R}\diagup2\pi\mathbb{Z},
\end{align}
and the $4$-ary multiplication is%
\begin{align}
&  \mu^{\left[  \left[  4\right]  \right]  }\left[  Q\left(  \alpha_{1}%
,\beta_{1},\gamma_{1}\right)  ,Q\left(  \alpha_{2},\beta_{2},\gamma
_{2}\right)  ,Q\left(  \alpha_{3},\beta_{3},\gamma_{3}\right)  ,Q\left(
\alpha_{4},\beta_{4},\gamma_{4}\right)  \right] \nonumber\\
&  =Q\left(  \alpha_{1},\beta_{1},\gamma_{1}\right)  Q\left(  \alpha_{2}%
,\beta_{2},\gamma_{2}\right)  Q\left(  \alpha_{3},\beta_{3},\gamma_{3}\right)
Q\left(  \alpha_{4},\beta_{4},\gamma_{4}\right) \nonumber\\
&  =Q\left(  \alpha_{1}+\beta_{2}+\gamma_{3}+\alpha_{4},\beta_{1}+\gamma
_{2}+\alpha_{3}+\beta_{4},\gamma_{1}+\alpha_{2}+\beta_{3}+\gamma_{4}\right)
=Q\left(  \alpha,\beta,\gamma\right)  , \label{mqa}%
\end{align}
which is noncommutative, as opposed to the binary product of $B$-matrices
(\ref{bba}).

The querelement $\overline{Q}\left(  \alpha,\beta,\gamma\right)  $ for a given
$Q\left(  \alpha,\beta,\gamma\right)  $ is defined by the equation (see
(\ref{qqq}))%
\begin{equation}
Q\left(  \alpha,\beta,\gamma\right)  Q\left(  \alpha,\beta,\gamma\right)
Q\left(  \alpha,\beta,\gamma\right)  \overline{Q}\left(  \alpha,\beta
,\gamma\right)  =Q\left(  \alpha,\beta,\gamma\right)  ,
\end{equation}
which has the solution%
\begin{equation}
\overline{Q}\left(  \alpha,\beta,\gamma\right)  =Q\left(  -\beta
-\gamma,-\alpha-\gamma,-\alpha-\beta\right)  . \label{qab}%
\end{equation}

\begin{definition}
We call $SO^{\left[  \left[  4\right]  \right]  }\left(  2,\mathbb{R}\right)
$ a polyadic ($4$-ary) special orthogonal group, and $Q\left(  \alpha
,\beta,\gamma\right)  $ is called a polyadic ($4$-ary) rotation matrix.
\end{definition}

Informally, the matrix $Q\left(  \alpha,\beta,\gamma\right)  $ represents the
polyadic ($4$-ary) rotation. There are an infinite number of polyadic
($4$-ary) identities (neutral elements) $E\left(  \alpha,\beta,\gamma\right)
$ which are defined by%
\begin{equation}
E\left(  \alpha,\beta,\gamma\right)  E\left(  \alpha,\beta,\gamma\right)
E\left(  \alpha,\beta,\gamma\right)  Q\left(  \alpha,\beta,\gamma\right)
=Q\left(  \alpha,\beta,\gamma\right)  , \label{eq}%
\end{equation}
and the solution is%
\begin{equation}
E\left(  \alpha,\beta,\gamma\right)  =Q\left(  \alpha,\beta,\gamma\right)
,\ \ \ \alpha+\beta+\gamma=0.
\end{equation}
It follows from (\ref{eq}) that $E\left(  \alpha,\beta,\gamma\right)  $ are
$4$-ary idempotents (see (\ref{ee}) and \textit{Remark} \ref{rem-idemp}).

The determinants of $B\left(  \alpha\right)  $ and $Q\left(  \alpha
,\beta,\gamma\right)  $ are $1$, and therefore the corresponding
multiplicative characters and polyadized multiplicative characters (\ref{xq})
are also equal to $1$.

Comparing with the successive products of four $B$-matrices (\ref{ba})%
\begin{equation}
B\left(  \alpha_{1}\right)  B\left(  \alpha_{2}\right)  B\left(  \alpha
_{3}\right)  B\left(  \alpha_{4}\right)  =B\left(  \alpha_{1}+\alpha
_{2}+\alpha_{3}+\alpha_{4}\right)  , \label{bbb}%
\end{equation}
we observe that $4$-ary multiplication (\ref{mqa}) gives a shifted sum of four angles.

More exactly, for the triple $\left(  \alpha,\beta,\gamma\right)  $ we
introduce the circle (left) shift operator by%
\begin{equation}
\mathsf{s}\alpha=\beta,\ \text{\ \ \ }\ \ \mathsf{s}\beta=\gamma
,\ \ \ \ \ \ \mathsf{s}\gamma=\alpha\label{sa}%
\end{equation}
with the property $\mathsf{s}^{3}=\mathsf{id}$. Then the $4$-ary
multiplication (\ref{mqa}) becomes%
\begin{align}
&  \mu^{\left[  \left[  4\right]  \right]  }\left[  Q\left(  \alpha_{1}%
,\beta_{1},\gamma_{1}\right)  ,Q\left(  \alpha_{2},\beta_{2},\gamma
_{2}\right)  ,Q\left(  \alpha_{3},\beta_{3},\gamma_{3}\right)  ,Q\left(
\alpha_{4},\beta_{4},\gamma_{4}\right)  \right] \nonumber\\
&  =Q\left(  \alpha_{1}+\mathsf{s}\alpha_{2}+\mathsf{s}^{2}\alpha_{3}%
+\alpha_{4},\beta_{1}+\mathsf{s}\beta_{2}+\mathsf{s}^{2}\beta_{3}+\beta
_{4},\gamma_{1}+\mathsf{s}\gamma_{2}+\mathsf{s}^{2}\gamma_{3}+\gamma
_{4}\right)  . \label{ms}%
\end{align}

The querelement has the form%
\begin{equation}
\overline{Q}\left(  \alpha,\beta,\gamma\right)  =Q\left(  -\mathsf{s}%
\alpha-\mathsf{s}^{2}\alpha,-\mathsf{s}\beta-\mathsf{s}^{2}\beta
,-\mathsf{s}\gamma-\mathsf{s}^{2}\gamma\right)  . \label{qs}%
\end{equation}

The multiplication (\ref{ms}) can be (informally) expressed in terms of a new
operation, the $4$-ary \textquotedblleft cyclic shift
addition\textquotedblright\ defined on $\mathbb{R}\times\mathbb{R}%
\times\mathbb{R}$ by (see (\ref{mqa}))%
\begin{align}
&  \mathbf{\nu}_{\mathsf{s}}^{\left[  4\right]  }\left[  \left(  \alpha
_{1},\beta_{1},\gamma_{1}\right)  ,\left(  \alpha_{2},\beta_{2},\gamma
_{2}\right)  ,\left(  \alpha_{3},\beta_{3},\gamma_{3}\right)  ,\left(
\alpha_{4},\beta_{4},\gamma_{4}\right)  \right] \nonumber\\
&  =\left(  \alpha_{1}+\beta_{2}+\gamma_{3}+\alpha_{4},\beta_{1}+\gamma
_{2}+\alpha_{3}+\beta_{4},\gamma_{1}+\alpha_{2}+\beta_{3}+\gamma_{4}\right)
\nonumber\\
&  =\left(  \nu_{\mathsf{s}}^{\left[  4\right]  }\left[  \alpha_{1},\alpha
_{2},\alpha_{3},\alpha_{4}\right]  ,\nu_{\mathsf{s}}^{\left[  4\right]
}\left[  \beta_{1},\beta_{2},\beta_{3},\beta_{4}\right]  ,\nu_{\mathsf{s}%
}^{\left[  4\right]  }\left[  \gamma_{1},\gamma_{2},\gamma_{3},\gamma
_{4}\right]  \right)  , \label{vss}%
\end{align}
where $\nu_{\mathsf{s}}^{\left[  4\right]  }$ is (informally)
\begin{equation}
\nu_{\mathsf{s}}^{\left[  4\right]  }\left[  \alpha_{1},\alpha_{2},\alpha
_{3},\alpha_{4}\right]  =\mathsf{s}^{0}\alpha_{1}+\mathsf{s}^{1}\alpha
_{2}+\mathsf{s}^{2}\alpha_{3}+\mathsf{s}^{3}\alpha_{4}=\alpha_{1}%
+\mathsf{s}\alpha_{2}+\mathsf{s}^{2}\alpha_{3}+\alpha_{4}, \label{va}%
\end{equation}
and $\mathsf{s}^{0}=\mathsf{id}$. This can also be treated as some
\textquotedblleft deformation\textquotedblright\ of the repeated binary
additions by shifts. It is seen that the $4$-ary operation $\mathbf{\nu
}_{\mathsf{s}}^{\left[  4\right]  }$ (\ref{vss}) is not derived and cannot be
obtained by consequent binary operations on the triples $\left(  \alpha
,\beta,\gamma\right)  $ as (\ref{bbb}).

In terms of the $4$-ary cyclic shift addition the $4$-ary multiplication
(\ref{ms}) becomes%
\begin{align}
&  \mu^{\left[  \left[  4\right]  \right]  }\left[  Q\left(  \alpha_{1}%
,\beta_{1},\gamma_{1}\right)  ,Q\left(  \alpha_{2},\beta_{2},\gamma
_{2}\right)  ,Q\left(  \alpha_{3},\beta_{3},\gamma_{3}\right)  ,Q\left(
\alpha_{4},\beta_{4},\gamma_{4}\right)  \right] \nonumber\\
&  =Q\left(  \mathbf{\nu}_{\mathsf{s}}^{\left[  4\right]  }\left[  \left(
\alpha_{1},\beta_{1},\gamma_{1}\right)  ,\left(  \alpha_{2},\beta_{2}%
,\gamma_{2}\right)  ,\left(  \alpha_{3},\beta_{3},\gamma_{3}\right)  ,\left(
\alpha_{4},\beta_{4},\gamma_{4}\right)  \right]  \right)  . \label{qv}%
\end{align}

The binary case corresponds to $\mathsf{s}=\mathsf{id}$, because in (\ref{ba})
we have only one angle $\alpha$, as opposed to three angles in (\ref{sa}).

Thus, we conclude that just as the binary product of $B$-matrices corresponds
to the ordinary angle addition (\ref{bba}), the $4$-ary multiplication of
polyadic rotation $Q$-matrices (\ref{qa}) corresponds to the $4$-ary cyclic
shift addition (\ref{va}) through (\ref{qv}).

\subsection{"Deformation" of binary operations by shifts}

The concrete example from the previous subsection shows the strong connection
(\ref{qv}) between the polyadization procedure and the shifted operations
(\ref{va}). Here we generalize it to an $n$-ary case for any semigroup.

Let $\mathcal{A}=\left\langle A\mid\left(  +\right)  \right\rangle $ be a
binary semigroup, where $A$ is its underlying set and $\left(  +\right)  $ is
the binary operation (which can be noncommutative). The simplest way to
construct an $n$-ary operation $\nu^{\left[  n\right]  }:A^{n}\rightarrow A$
is the consequent repetition of the binary operation (see (\ref{bbb}))%
\begin{equation}
\nu^{\left[  n\right]  }\left[  \alpha_{1},\alpha_{2},\ldots,\alpha
_{n}\right]  =\alpha_{1}+\alpha_{2}+\ldots+\alpha_{n}, \label{vn}%
\end{equation}
where the $n$-ary multiplication $\nu^{\left[  n\right]  }$ (\ref{vn}) is
called derived \cite{dor3,zup67}.

To construct a nonderived operation, we now consider the (external) $m$th
direct power $\mathcal{A}^{m}$ of the semigroup $\mathcal{A}$ by introducing
$m$-tuples%
\begin{equation}
\mathbf{a}\equiv\mathbf{a}^{\left(  m\right)  }=\overset{m}{\left(
\overbrace{\alpha,\beta,\ldots,\gamma}\right)  },\ \ \ \ \ \alpha,\beta
,\ldots,\gamma\in A,\ \ \ \ \mathbf{a}\in A^{m}. \label{a}%
\end{equation}

The $m$th direct power becomes a binary semigroup by endowing $m$-tuples with
the componentwise binary operation $\left(  \hat{+}\right)  $ as%
\begin{equation}
\mathbf{a}_{1}\hat{+}\mathbf{a}_{2}=\overset{m}{\left(  \overbrace{\alpha
_{1},\beta_{1},\ldots,\gamma_{1}}\right)  }\hat{+}\overset{m}{\left(
\overbrace{\alpha_{2},\beta_{2},\ldots,\gamma_{2}}\right)  }=\overset
{m}{\left(  \overbrace{\alpha_{1}+\alpha_{2},\beta_{1}+\beta_{2},\ldots
,\gamma_{1}+\gamma_{2}}\right)  }.
\end{equation}

The derived $n$-ary operation for $m$-tuples (on the $m$th direct power) is
then defined componentwise by analogy with (\ref{vn})%
\begin{equation}
\mathbf{\nu}^{\left[  n\right]  }\left[  \mathbf{a}_{1},\mathbf{a}_{2}%
,\ldots,\mathbf{a}_{n}\right]  =\mathbf{a}_{1}\hat{+}\mathbf{a}_{2}\hat
{+}\ldots\hat{+}\mathbf{a}_{n}. \label{aa}%
\end{equation}

Now using shifts, instead of (\ref{aa}) we construct a nonderived $n$-ary
operation on the direct power.

\begin{definition}
A cyclic $m$-shift operator $\mathsf{s}$ is defined for the $m$-tuple
(\ref{a}) by
\begin{equation}
\overset{m}{\overbrace{\mathsf{s}\alpha=\beta,\mathsf{s}\beta=\gamma
,\ \ldots\ ,\mathsf{s}\gamma=\alpha}}, \label{sab}%
\end{equation}
and $\mathsf{s}^{m}=\mathsf{id}$.
\end{definition}

For instance, in this notation, if $m=3$ and $\mathbf{a}=\left(  \alpha
,\beta,\gamma\right)  $, then $\mathsf{s}\mathbf{a=}\left(  \gamma
,\alpha,\beta\right)  $, $\mathsf{s}^{2}\mathbf{a=}\left(  \beta,\gamma
,\alpha\right)  $, $\mathsf{s}^{3}\mathbf{a=a}$ (as in the previous subsection).

To obtain a nonderived $n$-ary operation, by analogy with (\ref{vss}), we
deform by shifts the derived $n$-ary operation (\ref{aa}).

\begin{definition}
The shift deformation by (\ref{sab}) of the derived operation $\mathbf{\nu
}^{\left[  n\right]  }$ on the direct power $\mathcal{A}^{m}$ is defined
noncomponentwise by%
\begin{equation}
\mathbf{\nu}_{\mathsf{s}}^{\left[  n\right]  }\left[  \mathbf{a}%
_{1},\mathbf{a}_{2},\ldots,\mathbf{a}_{n}\right]  =%
{\displaystyle\sum\limits_{i=1}^{n}}
\mathsf{s}^{i-1}\mathbf{a}_{i}=\mathbf{a}_{1}\hat{+}\mathsf{s}\mathbf{a}%
_{2}\hat{+}\ldots\hat{+}\mathsf{s}^{n-1}\mathbf{a}_{n}, \label{vs}%
\end{equation}
where $\mathbf{a}\in A^{m}$ (\ref{a}) and $\mathsf{s}^{0}=\mathsf{id}$.
\end{definition}

Note that till now there exist no relations between $n$ and $m$.

\begin{proposition}
The shift deformed operation $\nu_{\mathsf{s}}^{\left[  n\right]  }$ is
totally associative, if%
\begin{align}
\mathsf{s}^{n-1}  &  =\mathsf{id},\label{sd}\\
m  &  =n-1.
\end{align}

\end{proposition}

\begin{proof}
We compute%
\begin{align}
&  \mathbf{\nu}_{\mathsf{s}}^{\left[  n\right]  }\left[  \mathbf{\nu
}_{\mathsf{s}}^{\left[  n\right]  }\left[  \mathbf{a}_{1},\mathbf{a}%
_{2},\ldots,\mathbf{a}_{n}\right]  ,\mathbf{a}_{n+1},\mathbf{a}_{n+2}%
,\ldots,\mathbf{a}_{2n-1}\right] \nonumber\\
&  =\left(  \mathbf{a}_{1}\hat{+}\mathsf{s}\mathbf{a}_{2}\hat{+}\ldots\hat
{+}\mathsf{s}^{n-1}\mathbf{a}_{n}\right)  \hat{+}\mathsf{s}\mathbf{a}%
_{n+1}\hat{+}\mathsf{s}^{2}\mathbf{a}_{n+2}\hat{+}\ldots\hat{+}\mathsf{s}%
^{n-1}\mathbf{a}_{2n-1}\nonumber\\
&  =\mathbf{a}_{1}\hat{+}\mathsf{s}\left(  \mathbf{a}_{2}\hat{+}%
\mathsf{s}\mathbf{a}_{3}\hat{+}\ldots\hat{+}\mathsf{s}^{n-1}\mathbf{a}%
_{n+1}\right)  \hat{+}\mathsf{s}^{2}\mathbf{a}_{n+2}\hat{+}\mathsf{s}%
^{3}\mathbf{a}_{n+3}\hat{+}\ldots\hat{+}\mathsf{s}^{n-1}\mathbf{a}%
_{2n-1}\nonumber\\
&  \vdots\nonumber\\
&  \mathbf{a}_{1}\hat{+}\mathsf{s}\mathbf{a}_{2}\hat{+}\ldots\hat{+}%
\mathsf{s}^{n-2}\mathbf{a}_{n}\hat{+}\mathsf{s}^{n-1}\left(  \mathbf{a}%
_{n+1}\hat{+}\mathsf{s}\mathbf{a}_{n+2}\hat{+}\mathsf{s}^{2}\mathbf{a}%
_{n+3}\hat{+}\ldots\hat{+}\mathsf{s}^{n-1}\mathbf{a}_{2n-1}\right) \nonumber\\
&  \mathbf{\nu}_{\mathsf{s}}^{\left[  n\right]  }\left[  \mathbf{a}%
_{1},\mathbf{a}_{2},\ldots,\mathbf{a}_{n-1},\mathbf{\nu}_{\mathsf{s}}^{\left[
n\right]  }\left[  \mathbf{a}_{n},\mathbf{a}_{n+1},\mathbf{a}_{n+2}%
,\ldots,\mathbf{a}_{2n-1}\right]  \right]  ,
\end{align}
which satisfy in all lines, if $\mathsf{s}^{n-1}=\mathsf{id}$ (\ref{sd}).
\end{proof}

\begin{corollary}
The set of $\left(  n-1\right)  $-tuples (\ref{a}) with the shift deformed
associative operation (\ref{vs}) is a nonderived $n$-ary semigroup
$\mathcal{S}_{shift}^{\left[  n\right]  }=\left\langle \left\{  \mathbf{a}%
\right\}  \mid\mathbf{\nu}_{\mathsf{s}}^{\left[  n\right]  }\right\rangle $
constructed from the binary semigroup $\mathcal{A}$.
\end{corollary}

\begin{proposition}
If the binary semigroup $\mathcal{A}$ is commutative, then $\mathcal{S}%
_{shift}^{\left[  n\right]  }$ becomes a nonderived $n$-ary group
$\mathcal{G}_{shift}^{\left[  n\right]  }=\left\langle \left\{  \mathbf{a}%
\right\}  \mid\mathbf{\nu}_{\mathsf{s}}^{\left[  n\right]  },\mathbf{\bar{\nu
}}_{\mathsf{s}}^{\left[  1\right]  }\right\rangle $, such that each element
$\mathbf{a}\in\mathcal{A}^{n-1}$ has a unique querelement $\mathbf{\bar{a}}$
(an analog of inverse) by%
\begin{equation}
\mathbf{\bar{a}}=\mathbf{\bar{\nu}}_{\mathsf{s}}^{\left[  1\right]  }\left[
\mathbf{a}\right]  =-\left(  \mathsf{s}\mathbf{a}\hat{+}\mathsf{s}%
^{2}\mathbf{a}\hat{+}\ldots\hat{+}\mathsf{s}^{n-2}\mathbf{a}\right)  ,
\label{asa}%
\end{equation}
where $\mathbf{\bar{\nu}}_{\mathsf{s}}^{\left[  1\right]  }:\mathcal{A}%
^{n-1}\rightarrow\mathcal{A}^{n-1}$ is an unary queroperation.
\end{proposition}

\begin{proof}
We have the definition of the querelement%
\begin{equation}
\mathbf{\nu}_{\mathsf{s}}^{\left[  n\right]  }\left[  \mathbf{\bar{a}%
},\mathbf{a},\ldots,\mathbf{a}\right]  =\mathbf{a,} \label{v1}%
\end{equation}
where $\mathbf{\bar{a}}$ can be on any place. So (\ref{vs}) gives the equation%
\begin{equation}
\mathbf{\bar{a}}\hat{+}\mathsf{s}\mathbf{a}\hat{+}\mathsf{s}^{2}\mathbf{a}%
\hat{+}\ldots\hat{+}\mathsf{s}^{n-2}\mathbf{a}\hat{+}\mathbf{a}=\mathbf{a,}%
\end{equation}
which can be resolved for the commutative and cancellative semigroup
$\mathcal{A}$ only, and the solution is (\ref{asa}). If $\mathbf{\bar{a}}$ is
on the $i$th place in (\ref{v1}), then it has the coefficient $\mathsf{s}%
^{i-1}$, and we multiply both sides by $\mathsf{s}^{n-i}$ to get
$\mathbf{\bar{a}}$ without any shift operator coefficient using (\ref{sd}),
which gives the same solution (\ref{asa}).
\end{proof}

For $n=4$ and $\mathbf{a}=\left(  \alpha,\beta,\gamma\right)  $, the equation
(\ref{v1}) is%
\begin{equation}
\mathbf{\bar{a}}\hat{+}\mathsf{s}\mathbf{a}\hat{+}\mathsf{s}^{2}\mathbf{a}%
\hat{+}\mathbf{a}=\mathbf{a}%
\end{equation}
and (see(\ref{qs}))%
\begin{equation}
\mathbf{\bar{a}}=-\left(  \mathsf{s}\mathbf{a}\hat{+}\mathsf{s}^{2}%
\mathbf{a}\right)
\end{equation}
so (cf. (\ref{qab}))%
\begin{equation}
\mathbf{\bar{a}}=\overline{\left(  \alpha,\beta,\gamma\right)  }%
=\mathbf{-}\left(  \gamma+\beta,\alpha+\gamma,\beta+\alpha\right)  .
\end{equation}

It is known that the existence of an identity (as a neutral element) is not
necessary for polyadic groups, and only a querelement is important
\cite{dor3,gle/gla}. Nevertheless, we have

\begin{proposition}
If the commutative and cancellative semigroup $\mathcal{A}$ has zero $0\in A$,
then the $n$-ary group $\mathcal{G}_{shift}^{\left[  n\right]  }$ has a set of
polyadic identities (idempotents) satisfying%
\begin{equation}
\mathbf{e}\hat{+}\mathsf{s}\mathbf{e}\hat{+}\ldots\hat{+}\mathsf{s}%
^{n-2}\mathbf{e}=\mathbf{0}, \label{es}%
\end{equation}
where $\mathbf{0}=\overset{n-1}{\left(  \overbrace{0,0,\ldots,0}\right)  }$ is
the zero $\left(  n-1\right)  $-tuple.
\end{proposition}

\begin{proof}
The definition of polyadic identity in terms of the deformed $n$-ary product
in the direct power is%
\begin{equation}
\mathbf{\nu}_{\mathsf{s}}^{\left[  n\right]  }\left[  \overset{n-1}%
{\overbrace{\mathbf{e},\mathbf{e},\ldots,\mathbf{e}}},\mathbf{a}\right]
=\mathbf{a,}\ \ \ \forall\mathbf{a}\in\mathcal{A}^{n-1}. \label{ve}%
\end{equation}
Using (\ref{vs}) we get the equation%
\begin{equation}
\mathbf{e}\hat{+}\mathsf{s}\mathbf{e}\hat{+}\mathsf{s}^{2}\mathbf{e}\hat
{+}\ldots\hat{+}\mathsf{s}^{n-2}\mathbf{e}\hat{+}\mathbf{a}=\mathbf{a.}%
\end{equation}
After cancellation by $\mathbf{a}$ we obtain (\ref{es}).
\end{proof}

For $n=4$ and $\mathbf{e}=\left(  \alpha_{0},\beta_{0},\gamma_{0}\right)  $ we
obtain an infinite set of identities satisfying%
\begin{equation}
\mathbf{e}=\left(  \alpha_{0},\beta_{0},\gamma_{0}\right)  ,\ \ \ \alpha
_{0}+\beta_{0}+\gamma_{0}=0.
\end{equation}
To see that they are $4$-ary idempotents, insert $\mathbf{a}=\mathbf{e}$ into
(\ref{ve}).

Thus, starting from a binary semigroup $\mathcal{A}$, using our polyadization
procedure we have obtained a nonderived $n$-ary group on $\left(  n-1\right)
$th direct power $\mathcal{A}^{n-1}$ with the shift deformed multiplication.
This construction reminds the Post-like associative quiver from
\cite{dup2018a,duplij2022}, and allows us to construct a nonderived $n$-ary
group from any semigroup in the unified way presented here.

\subsection{Polyadization of binary supergroups}

Here we consider a more exotic possibility, when the $\mathrm{B}$-matrices are
defined over the Grassmann algebra, and therefore can represent supergroups
(see (\ref{ws}) and below). In this case $\mathrm{B}$'s can be supermatrices
of two kinds, even and odd, which have different properties
\cite{berezin,leites}. The general polyadization procedure remains the same,
as for ordinary matrices considered before (see \textbf{Definition
\ref{def-pol}}), and therefore we confine ourselves to examples only.

Indeed, to obtain an $n$-ary matrix (semi)group represented now by the
$\mathrm{Q}$-supermatrices (\ref{q}) of the nonstandard form, we should take
$\left(  n-1\right)  $ initial $\mathrm{B}$-supermatrices which present a
simple ($k=1$ in (\ref{n2})) binary (semi)supergroup, which now have different
parameters $\mathrm{B}_{\mathbf{y}_{i}}\equiv\mathrm{B}_{\mathbf{y}_{i}%
}\left(  \left(  p_{even}\mid p_{odd}\right)  \times\left(  p_{even}\mid
p_{odd}\right)  \right)  $, $i=1,\ldots,n-1$, where $p_{even}$ and $p_{odd}$
are even and odd dimensions of the $\mathrm{B}$-supermatrix. The closure of
the $\mathrm{Q}$-supermatrix multiplication is governed by the closure of
$\mathrm{B}$-supermatrix multiplication (\ref{b1})--(\ref{b2}) in the initial
binary (semi)supergroup.

\subsubsection{Polyadization of $GL\left(  1\mid1,\Lambda\right)  $}

Let $\Lambda=\Lambda_{even}\oplus\Lambda_{odd}$ be a Grassmann algebra over
$\mathbb{C}$, where $\Lambda_{even}$ and $\Lambda_{odd}$ are its even and odd
parts (it can be also any commutative superalgebra). We provide (in brief) the
polyadization procedure of the general linear supergroup $GL\left(
1\mid1,\Lambda\right)  $ for $n=3$. The $4$-parameter block (invertible)
supermatrices become $\mathrm{B}_{\mathbf{y}_{i}}=\left(
\begin{array}
[c]{cc}%
a_{i} & \alpha_{i}\\
\beta_{i} & b_{i}%
\end{array}
\right)  \in GL\left(  1\mid1,\Lambda\right)  $, where the parameters are
$\mathbf{y}_{i}=\left(  a_{i},b_{i},\alpha_{i},\beta_{i}\right)  \in
\Lambda_{even}\times\Lambda_{even}\times\Lambda_{odd}\times\Lambda_{odd},$
$i=1,2$. Thus, the $8$-parameter ternary supergroup $GL^{\left[  \left[
3\right]  \right]  }\left(  1\mid1,\Lambda\right)  =\left\langle \left\{
\mathrm{Q}_{\mathbf{y}_{1},\mathbf{y}_{2}}\right\}  \mid\mu^{\left[  \left[
3\right]  \right]  }\right\rangle $ is represented by the following $4\times4$
$\mathrm{Q}$-supermatrices%
\begin{equation}
\mathrm{Q}_{\mathbf{y}_{1},\mathbf{y}_{2}}=\left(
\begin{array}
[c]{cc}%
0 & \mathrm{B}_{\mathbf{y}_{1}}\\
\mathrm{B}_{\mathbf{y}_{2}} & 0
\end{array}
\right)  =\left(
\begin{array}
[c]{cccc}%
0 & 0 & a_{1} & \alpha_{1}\\
0 & 0 & \beta_{1} & b_{1}\\
a_{2} & \alpha_{2} & 0 & 0\\
\beta_{2} & b_{2} & 0 & 0
\end{array}
\right)  \in GL^{\left[  \left[  3\right]  \right]  }\left(  1\mid
1,\Lambda\right)  ,
\end{equation}
which satisfy the ternary (nonderived) multiplication%
\begin{equation}
\mu^{\left[  \left[  3\right]  \right]  }\left[  \mathrm{Q}_{\mathbf{y}%
_{1}^{\prime},\mathbf{y}_{2}^{\prime}},\mathrm{Q}_{\mathbf{y}_{1}%
^{\prime\prime},\mathbf{y}_{2}^{\prime\prime}},\mathrm{Q}_{\mathbf{y}%
_{1}^{\prime\prime\prime},\mathbf{y}_{2}^{\prime\prime\prime}}\right]
=\mathrm{Q}_{\mathbf{y}_{1}^{\prime},\mathbf{y}_{2}^{\prime}}\mathrm{Q}%
_{\mathbf{y}_{1}^{\prime\prime},\mathbf{y}_{2}^{\prime\prime}}\mathrm{Q}%
_{\mathbf{y}_{1}^{\prime\prime\prime},\mathbf{y}_{2}^{\prime\prime\prime}%
}=\mathrm{Q}_{\mathbf{y}_{1},\mathbf{y}_{2}}.
\end{equation}
In terms of the block matrices $\mathrm{B}_{\mathbf{y}_{i}}$ the
multiplication (\ref{mq4}) becomes (see (\ref{mq})--(\ref{b2}))%
\begin{align}
\mathrm{B}_{\mathbf{y}_{1}^{\prime}}\mathrm{B}_{\mathbf{y}_{2}^{\prime\prime}%
}\mathrm{B}_{\mathbf{y}_{1}^{\prime\prime\prime}}  &  =\mathrm{B}%
_{\mathbf{y}_{1}},\\
\mathrm{B}_{\mathbf{y}_{2}^{\prime}}\mathrm{B}_{\mathbf{y}_{1}^{\prime
\prime\prime}}\mathrm{B}_{\mathbf{y}_{2}^{\prime\prime\prime}}  &
=\mathrm{B}_{\mathbf{y}_{2}}.
\end{align}
In terms of the $B$-supermatrix parameters the supergroup $GL^{\left[  \left[
3\right]  \right]  }\left(  1\mid1,\Lambda\right)  $ is defined by%
\begin{align}
\alpha_{1}^{\prime}\beta_{2}^{\prime\prime}a_{1}^{\prime\prime\prime}%
+a_{1}^{\prime}\alpha_{2}^{\prime\prime}\beta_{1}^{\prime\prime\prime}%
+\alpha_{1}^{\prime}b_{2}^{\prime\prime}\beta_{1}^{\prime\prime\prime}%
+a_{1}^{\prime}a_{2}^{\prime\prime}a_{1}^{\prime\prime\prime}  &
=a_{1},\ \ \beta_{1}^{\prime}a_{2}^{\prime\prime}\alpha_{1}^{\prime
\prime\prime}+\beta_{1}^{\prime}\alpha_{2}^{\prime\prime}b_{1}^{\prime
\prime\prime}+b_{1}^{\prime}\beta_{2}^{\prime\prime}\alpha_{1}^{\prime
\prime\prime}+b_{1}^{\prime}b_{2}^{\prime\prime}b_{1}^{\prime\prime\prime
}=b_{1},\nonumber\\
\alpha_{1}^{\prime}\beta_{2}^{\prime\prime}\alpha_{1}^{\prime\prime\prime
}+a_{1}^{\prime}a_{2}^{\prime\prime}\alpha_{1}^{\prime\prime\prime}%
+a_{1}^{\prime}\alpha_{2}^{\prime\prime}b_{1}^{\prime\prime\prime}+\alpha
_{1}^{\prime}b_{2}^{\prime\prime}b_{1}^{\prime\prime\prime}  &  =\alpha
_{1},\ \ \beta_{1}^{\prime}\alpha_{2}^{\prime\prime}\beta_{1}^{\prime
\prime\prime}+\beta_{1}^{\prime}a_{2}^{\prime\prime}a_{1}^{\prime\prime\prime
}+b_{1}^{\prime}\beta_{2}^{\prime\prime}a_{1}^{\prime\prime\prime}%
+b_{1}^{\prime}b_{2}^{\prime\prime}\beta_{1}^{\prime\prime\prime}=\beta
_{1},\nonumber\\
\alpha_{2}^{\prime}\beta_{1}^{\prime\prime}a_{2}^{\prime\prime\prime}%
+a_{2}^{\prime}\alpha_{1}^{\prime\prime}\beta_{2}^{\prime\prime\prime}%
+\alpha_{2}^{\prime}b_{1}^{\prime\prime}\beta_{2}^{\prime\prime\prime}%
+a_{2}^{\prime}a_{1}^{\prime\prime}a_{2}^{\prime\prime\prime}  &
=a_{2},\ \ \beta_{2}^{\prime}a_{1}^{\prime\prime}\alpha_{2}^{\prime
\prime\prime}+\beta_{2}^{\prime}\alpha_{1}^{\prime\prime}b_{2}^{\prime
\prime\prime}+b_{2}^{\prime}\beta_{1}^{\prime\prime}\alpha_{2}^{\prime
\prime\prime}+b_{2}^{\prime}b_{1}^{\prime\prime}b_{2}^{\prime\prime\prime
}=b_{2},\nonumber\\
\alpha_{2}^{\prime}\beta_{1}^{\prime\prime}\alpha_{2}^{\prime\prime\prime
}+a_{2}^{\prime}a_{1}^{\prime\prime}\alpha_{2}^{\prime\prime\prime}%
+a_{2}^{\prime}\alpha_{1}^{\prime\prime}b_{2}^{\prime\prime\prime}+\alpha
_{2}^{\prime}b_{1}^{\prime\prime}b_{2}^{\prime\prime\prime}  &  =\alpha
_{2},\ \ \beta_{2}^{\prime}\alpha_{1}^{\prime\prime}\beta_{2}^{\prime
\prime\prime}+\beta_{2}^{\prime}a_{1}^{\prime\prime}a_{2}^{\prime\prime\prime
}+b_{2}^{\prime}\beta_{1}^{\prime\prime}a_{2}^{\prime\prime\prime}%
+b_{2}^{\prime}b_{1}^{\prime\prime}\beta_{2}^{\prime\prime\prime}=\beta_{2}.
\end{align}

The unique querelement in $GL^{\left[  \left[  3\right]  \right]  }\left(
1\mid1,\Lambda\right)  $ can be found from the equation (see (\ref{mqq}))%
\begin{equation}
\mathrm{Q}_{\mathbf{y}_{1},\mathbf{y}_{2}}\mathrm{Q}_{\mathbf{y}%
_{1},\mathbf{y}_{2}}\overline{\mathrm{Q}}_{\mathbf{y}_{1},\mathbf{y}_{2}%
}=\mathrm{Q}_{\mathbf{y}_{1},\mathbf{y}_{2}},
\end{equation}
where the solution is%
\begin{equation}
\overline{\mathrm{Q}}_{\mathbf{y}_{1},\mathbf{y}_{2}}=\left(
\begin{array}
[c]{cc}%
0 & \overline{\mathrm{B}}_{\mathbf{y}_{1}}\\
\overline{\mathrm{B}}_{\mathbf{y}_{2}} & 0
\end{array}
\right)  ,
\end{equation}
with (see (\ref{bb1}))%
\begin{equation}
\overline{\mathrm{B}}_{\mathbf{y}_{1}}=\mathrm{B}_{\mathbf{y}_{2}}%
^{-1},\ \ \ \ \overline{\mathrm{B}}_{\mathbf{y}_{2}}=\mathrm{B}_{\mathbf{y}%
_{1}}^{-1},
\end{equation}
and $\mathrm{B}_{\mathbf{y}_{1}}^{-1},\mathrm{B}_{\mathbf{y}_{2}}^{-1}\in
GL\left(  1\mid1,\Lambda\right)  $.

\begin{definition}
We call $GL^{\left[  \left[  3\right]  \right]  }\left(  1\mid1,\Lambda
\right)  $ a polyadic (ternary) general linear supergroup obtained by the
polyadization procedure from the binary linear supergroup $GL\left(
1\mid1,\Lambda\right)  $.
\end{definition}

The ternary identity $\mathrm{E}_{4}^{\left(  3\right)  }$ of $GL^{\left[
\left[  3\right]  \right]  }\left(  1\mid1,\Lambda\right)  $ has the form (see
(\ref{e}))%
\begin{equation}
\mathrm{E}_{4}^{\left(  3\right)  }=\left(
\begin{array}
[c]{cc}%
0 & \mathrm{E}_{2}\\
\mathrm{E}_{2} & 0
\end{array}
\right)  ,
\end{equation}
where $\mathrm{E}_{2}$ is the identity of $GL\left(  1\mid1,\Lambda\right)  $,
and is ternary idempotent%
\begin{equation}
\mathrm{E}_{4}^{\left(  3\right)  }\mathrm{E}_{4}^{\left(  3\right)
}\mathrm{E}_{4}^{\left(  3\right)  }=\mathrm{E}_{4}^{\left(  3\right)  }.
\end{equation}

The ternary supergroup $GL^{\left[  \left[  3\right]  \right]  }\left(
1\mid1,\Lambda\right)  $ contains the infinite number of ternary idempotents
$\mathrm{Q}_{\mathbf{y}_{1},\mathbf{y}_{2}}^{idemp}$ defined by the system of
equations%
\begin{equation}
\mathrm{Q}_{\mathbf{y}_{1},\mathbf{y}_{2}}^{idemp}\mathrm{Q}_{\mathbf{y}%
_{1},\mathbf{y}_{2}}^{idemp}\mathrm{Q}_{\mathbf{y}_{1},\mathbf{y}_{2}}%
^{idemp}=\mathrm{Q}_{\mathbf{y}_{1},\mathbf{y}_{2}}^{idemp},
\end{equation}
which gives%
\begin{equation}
\mathrm{B}_{\mathbf{y}_{1}}^{idemp}\mathrm{B}_{\mathbf{y}_{2}}^{idemp}%
=\mathrm{E}_{2}. \label{bbe}%
\end{equation}

Therefore, the idempotents are determined by $8-4=4$ Grassmann parameters. One
of the ways to realize this is to exclude from (\ref{bbe}) the $2\times2$
$\mathrm{B}$-supermatrix. In this case, the idempotents in the supergroup
$GL^{\left[  \left[  3\right]  \right]  }\left(  1\mid1,\Lambda\right)  $
become%
\begin{equation}
\mathrm{Q}_{\mathbf{y}_{1},\mathbf{y}_{2}}^{idemp}=\left(
\begin{array}
[c]{cc}%
0 & \mathrm{B}_{\mathbf{y}_{1}}\\
\left(  \mathrm{B}_{\mathbf{y}_{1}}\right)  ^{-1} & 0
\end{array}
\right)  ,
\end{equation}
where $\mathrm{B}_{\mathbf{y}_{i}}\in GL\left(  1\mid1,\Lambda\right)  $ is an
invertible $2\times2$ supermatrix of the standard form (see \textit{Remark
}\ref{rem-idemp}).

In the same way one can polyadize any supergroup that can be presented by supermatrices.

\section{\textsc{Conclusions}}

In this paper we have given answers to the following important questions: how
to obtain nonderived polyadic structures from binary ones, and what would be a
matrix form of their semisimple versions? First, we introduced a general
matrix form for polyadic structures in terms of block-shift matrices. If the
blocks correspond to a binary structure (a ring, semigroup, group or
supergroup), this can be treated as a polyadization procedure for them.
Second, the semisimple blocks which further have a block-diagonal form give
rise to semisimple nonderived polyadic structures. For a deeper and expanded
understanding of the new constructions introduced, we have given clarifying
examples. The polyadic structures presented can be used, e.g. for the further
development of differential geometry and operad theory, as well as in other
directions which use higher arity and nontrivial properties of the constituent
universal objects.\bigskip

\textbf{Acknowledgements.} The author is grateful to Vladimir Akulov, Mike
Hewitt, Dimitrij Leites, Vladimir Tkach, Raimund Vogl and Alexander Voronov
for useful discussions, and valuable help.

\pagestyle{emptyf}
\mbox{}

\end{document}